\documentclass[11pt]{amsart}
\usepackage{amscd,amssymb,euscript,hyperref}

\usepackage[OT2, T1]{fontenc}

\renewcommand{\phi}{\varphi}
\newcommand\A{\mathbb A}
\renewcommand\P{\mathbb P}

\newcommand\gm[1]{\mathbb{G}_{\mathrm{m},#1}}
\newcommand\gmR{\gm{R}}
\newcommand\gmW{\gm{W}}

\DeclareMathOperator\Ru{R_u}
\DeclareMathOperator\Glue{{Glue}}

\DeclareMathOperator\Pic{{Pic}}
\DeclareMathOperator\Spec{{Spec}}

\DeclareMathOperator{\Frac}{Frac}
\DeclareMathOperator{\ad}{ad}

\DeclareMathOperator{\Hom}{Hom}

\DeclareMathOperator{\GL}{GL}

\DeclareMathOperator{\gl}{\mathfrak{gl}}
\DeclareMathOperator{\SL}{SL}

\DeclareMathOperator{\Aut}{Aut}
\DeclareMathOperator{\Iso}{Iso}

\DeclareMathOperator{\Res}{Res}

\newcommand{\fm}{\mathfrak m}

\renewcommand{\sc}{\mathrm{sc}}

\newcounter{noindnum}[subsection]
\renewcommand{\thenoindnum}{\alph{noindnum}}
\newcommand{\noindstep}{\refstepcounter{noindnum}{\rm(}\thenoindnum\/{\rm)} }
\newcommand{\stepzero}{\setcounter{noindnum}{0}
}
\theoremstyle{plain}
\newtheorem{theorem}{Theorem}
\newtheorem{proposition}{Proposition}[section]
\newtheorem{lemma}[proposition]{Lemma}
\newtheorem{corollary}[theorem]{Corollary}

\theoremstyle{definition}
\newtheorem{definition}[proposition]{Definition}

\theoremstyle{remark}
\newtheorem{remark}[proposition]{Remark}
\newtheorem*{remark*}{Remark}
\newtheorem{remarks}[proposition]{Remarks}
\newtheorem*{remarks*}{Remarks}

\newtheorem*{conventions*}{Conventions}

\newcommand\fp{{\mathfrak{p}}}

\newcommand\cB{{\mathcal{B}}}

\newcommand\cE{{\mathcal{E}}}
\newcommand\cF{{\mathcal{F}}}
\newcommand\cG{{\mathcal{G}}}
\newcommand\cH{{\mathcal{H}}}

\newcommand\cO{{\mathcal{O}}}

\newcommand\cT{{\mathcal{T}}}

\newcommand\cV{{\mathcal{V}}}
\newcommand\cP{{\mathcal{P}}}

\newcommand\bA{{\mathbf A}}
\newcommand\bB{{\mathbf B}}

\newcommand\bT{{\mathbf T}}
\newcommand\bG{{\mathbf G}}
\newcommand\bP{{\mathbf P}}

\newcommand\bH{{\mathbf H}}

\newcommand\bU{{\mathbf U}}

\newcommand\bZ{{\mathbf Z}}

\newcommand\bx{{\mathbf x}}

\newcommand\Z{\mathbb Z}

\title[Grothendieck--Serre for isotropic group schemes]{Unramified Grothendieck--Serre for simply-connected group schemes satisfying an isotropy condition via unipotent chains}
\author{Roman Fedorov}
\email{fedorov@pitt.edu}
\address{University of Pittsburgh, Pittsburgh, PA}
\address{Max Planck Institute for Mathematics, Bonn, Germany}

\begin{document}

\begin{abstract}
We prove a case of the Grothendieck--Serre conjecture: let $R$ be a Noetherian semilocal flat algebra over a Dedekind domain such that all fibers of $R$ are geometrically regular; let $\mathbf G$ be a~simply-connected reductive $R$-group scheme having a strictly proper parabolic subgroup scheme. Then a~$\mathbf G$-torsor over $R$ is trivial, provided that it is trivial over the total ring of fractions of $R$. We also simplify the proof of the conjecture in the quasi-split unramified case. The argument is based on the~notion of a~unipotent chain of torsors that we introduce. We also prove that if $R$ is a Noetherian normal domain and $\mathbf G$ is as above, then for any generically trivial torsor over an open subset $U$ of $\mathop{\mathrm{Spec}}R$, there is a closed $Z\subset\mathop{\mathrm{Spec}}R$ of codimension at least two such the torsor trivializes over every affine scheme that factors through $U-Z$.
\end{abstract}

\maketitle

\section{Introduction and Main results}
\subsection{} Let $R$ be a regular local ring; let $\bG$ be a reductive group scheme over $R$. A~conjecture of Grothendieck and Serre (see~\cite[remarque, p.31]{SerreFibres}, \cite[remarque~3, p.26-27]{GrothendieckTorsion}, and~\cite[remarque~1.11.a]{GrothendieckBrauer2}) predicts that a $\bG$-torsor over $R$ is trivial, if it is trivial over the fraction field of $R$. Recently this has been proved in the case when $R$ contains an infinite field in~\cite{FedorovPanin}, it was extended to the case of finite fields in~\cite{PaninFiniteFieldsIzvestiya}. In the mixed characteristic case the conjecture was previously known if $R$ is unramified (that is, the fibers of the projection $\Spec R\to\Spec\Z$ are regular) and $\bG$ is quasi-split, see~\cite{CesnaviciusGrSerre}. Assuming that the group scheme is simply-connected (so, necessarily semisimple), we generalize this result from the quasi-split case to the case of group schemes satisfying an isotropy condition. We also reprove the aforementioned result of~\cite{CesnaviciusGrSerre} (see Remark~\ref{rem:rem2}\eqref{rem:Ces}). Finally, we show that, for $\bG$ as above and any normal Noetherian domain $R$, every generically trivial $\bG$-torsor is ``almost trivial'' away from codimension at least two; see Section~\ref{sect:AlmostTrivial}.

Following~\cite{GilleStavrovaRequivalence}, we make the following definition (cf.~Remark~\ref{rem:rem}\eqref{rem:Total}).
\begin{definition}\label{def:StrictlyProper}
  A parabolic subgroup scheme $\bP$ of a reductive group scheme $\bG$ is called \emph{strictly proper,\/} if it intersects nontrivially every normal semisimple subgroup scheme of $\bG$.
\end{definition}

We prove the following theorem, which is a generalization of~\cite[Thm.~9.1]{CesnaviciusGrSerre} in the simply-connected case.
\begin{theorem}\label{th:Main}
Let $R$ be a Noetherian semilocal flat $\cO$-algebra, where $\cO$ is a semilocal Dedekind domain. Assume that all the fibers of $\Spec R\to\Spec\cO$ are geometrically regular. Let $\bG$ be a simply-connected reductive $R$-group scheme having a strictly proper parabolic $R$-subgroup scheme. Let $\cE$ be a $\bG$-torsor over~$R$. If $\cE$ is trivial over the total ring of fractions of $R$, then $\cE$ is trivial.
\end{theorem}

The theorem will be derived from Theorems~\ref{th:Chain} and~\ref{th:Main2} in Section~\ref{sect:Chains}.

\begin{remarks}\label{rem:rem}
\stepzero\noindstep\label{rem:regular} In the situation of Theorem~\ref{th:Main}, $R$ is automatically regular. Indeed, Popescu's Theorem~\cite[Tag07GC]{StacksProject} shows that $R$ is a filtered colimit of smooth finitely generated $\cO$-algebras, which are regular by~\cite[Tag07NF]{StacksProject}.

\noindstep\label{rem:Total} If $R$ is any local ring, then a reductive $R$-group scheme $\bG$ has a strictly proper parabolic subgroup scheme if and only if it is totally isotropic in the terminology of~\cite{CesnaviciusProblems} (or strongly locally isotropic in the terminology of~\cite{FedorovGrSerreNonSC}). If $R$ is only semi-local, one can imagine that a totally isotropic $\bG$ has a strictly proper parabolic $R$-subgroup scheme locally over $R$ but not globally. I expect that this cannot actually happen if $R$ is regular.

\noindstep In~\cite[Prop.~5.1]{FedorovCesnavicius} a stronger result is proved: under the same assumption on $R$, the Grothendieck--Serre conjecture holds for all totally isotropic reductive $R$-group schemes. The argument of~\cite{FedorovCesnavicius} is completely different. It is unlikely that Theorem~\ref{th:AlmostTrivial} below can be proved in full generality by the methods of~\cite{FedorovCesnavicius}.
\end{remarks}

Let $(R,\fm)$ be a regular local ring. Recall that $R$ is \emph{unramified\/} if $R/pR$ is regular, where $p$ is the characteristic of the residue field $R/\fm$. It follows from Theorem~\ref{th:Main} that the Grothendieck--Serre conjecture holds for simply-connected, totally isotropic reductive group schemes over unramified regular local rings:

\begin{corollary}\label{cor:GrSerre}
    Let $R$ be an unramified regular local ring and let $\bG$ be a simply-connected totally isotropic reductive $R$-group scheme. Then a $\bG$-torsor over $R$ is trivial, provided that it is trivial over the fraction field of $R$.
\end{corollary}
\begin{proof}
  If $R$ contains a field, then it coincides with its fiber $R/pR$, and the statement follows from~\cite{FedorovPanin,PaninFiniteFieldsIzvestiya}. Otherwise, $R$ is a flat $\cO$-algebra, where $\cO=\Z_{p\Z}$. The special fiber of the projection $\Spec R\to\Spec\cO$ is equal to $R/pR$; by definition of unramifiedness it is regular, thus it is geometrically regular over $\Z/p\Z$ because $\Z/p\Z$ is perfect. Since $R$ is regular, the generic fiber of $\Spec R\to\Spec\cO$ is regular (and thus geometrically regular). Hence we can apply Theorem~\ref{th:Main}.
\end{proof}

\subsection{``Almost trivial'' torsors}\label{sect:AlmostTrivial} We now explain the main idea of our argument (and, in fact, of most proofs of the cases of the Grothendieck--Serre conjecture in mixed characteristic). Using Popescu's Theorem, we may assume that $\bG$ and its generically trivial torsor $\cE$ are defined over an affine scheme $X$ that is smooth over a Dedekind domain $\cO$. We assume that $X$ is of positive dimension as $\cE$ is trivial otherwise. The $X$-torsor $\cE$, being generically trivial, is trivial away from a closed subset $T\subset X$ of codimension at least one. In the case when $\cO$ is a field (that is, in the equal characteristic case), this allows us to fiber $X$ (after shrinking it) into smooth curves over an open subset of an affine space in such a way that $T$ is finite over the base. As explained below in Section~\ref{step:Difficulty}, the main problem in mixed characteristic is that we ``loose one dimension,'' so we would need $T$ to be of codimension at least two. This is, however, impossible: by Hartogs principle the trivialization over $X-T$ would extend to $X$. However, we can find a closed subset $T$ of codimension at least two in $X$ such that $\cE$ is ``almost trivial'' over $X-T$ in the following sense:

\begin{theorem}\label{th:AlmostTrivial}
Let $R$ be a normal Noetherian domain and $\bG$ be a simply-connected reductive $R$-group scheme having a strictly proper parabolic $R$-subgroup scheme. Let $\cE$ be a generically trivial $\bG$-torsor over an open subscheme $U\subset\Spec R$. Then there is a closed subset $T\subset\Spec R$ of codimension at least two such that for any morphism $\Spec A\to U-T$ the pullback of $\cE$ to $\Spec A$ is trivial.
\end{theorem}

We emphasize that this theorem has very modest assumptions on $R$: it is only required to be normal and Noetherian. The author does not expect this statement to hold beyond the simply-connected isotropic case (e.g., even for tori). We also emphasize that $\cE$ is not required to be defined over $\Spec R$ but only over an open subset thereof.

\subsection{History and overview of the proof}\label{sect:Overview} Note that the Grothendieck--Serre conjecture was initially formulated for local rings but it has been generalized to semilocal rings; most of the papers cited deal with the semilocal case.

The conjecture was proved in the case when the semilocal ring contains an infinite field in 2012 by Fedorov and Panin who completed the work of many people; see~\cite{FedorovPanin} and the historical remarks therein. The conjecture was proved when the semilocal ring contains a finite field by Panin in 2014; see~\cite{PaninFiniteFieldsIzvestiya}. We briefly explain the strategy of the proof in the equicharacteristic case.

\noindstep\label{step:start} The first step is to use Popescu's Theorem to reduce to the case when $R$ is the semilocal ring of a finite set $\bx$ of closed points on an integral affine scheme $X$ smooth over a field $k$. By spreading out we may assume that the group scheme $\bG$ and the $\bG$-torsor $\cE$ are defined over $X$. We may also assume that~$\cE$ is trivial away from a proper closed subset $Y\subset X$.

\noindstep\label{step:fibering} The second step is to fiber an open neighborhood $X'\subset X$ of $\bx$ over an open subset $S\subset \A^{\dim X-1}_k$ in such a way that the fibers are smooth curves and $Y\cap X'$ is finite over $S$. In fact, any generic projection does the job. Formally, this is accomplished by compactifying $X$ and using Bertini's Theorem. As we will see momentarily, this is where the situation becomes drastically different in the mixed characteristic case. We note that the finite field case became approachable after finite field Bertini's Theorem became available (see~\cite{PoonenOnBertini}). It seems plausible that one can use results of~\cite{Gabber2001SpaceFilling} instead.

\noindstep\label{step:Descend} The third step is to replace $X'\to S$ with $C:=X'\times_S\Spec R\to\Spec R$ and $\cE$ with $\cE':=p_1^*\cE$. The original torsor $\cE$ is recovered as $\Delta^*\cE$, where $\Delta\colon\Spec R\to C$ is the diagonal section. The data $(C\to\Spec R, \Delta, \cE')$ can be improved without changing $\Delta^*\cE$ until $\cE'$ can be descended onto $\A^1_R$.

\noindstep\label{step:A1} The last step is to show that \emph{if $\cE'$ is a torsor over $\A^1_R$ trivial away from an $R$-finite subscheme, then $\cE'$ is trivial along every section of $\A^1_R\to\Spec R$}. We note that originally this statement was only available in the case when $\bG$ is simple and simply-connected, so one had to reduce to this case. One of the main results of~\cite{FedorovGrSerreNonSC} shows that the statement is true for all reductive group schemes $\bG$.

\subsubsection{}\label{step:Difficulty} Very little was known about the mixed characteristic case until~2015 (see~\cite{FedorovMixedChar} and the historical remarks in loc.~cit.) Here are the main ideas of loc.~cit. Assume that $R$ is unramified. Then the main difficulty is that the argument in~\eqref{step:fibering} fails. For example, if $\dim X=2$, then the fibers of the projection $X\to\Spec\Z_{p\Z}$ are already one-dimensional, so there is nothing to fiber. On the other hand, $Y$ may not be finite over $\Spec\Z_{p\Z}$, which is crucial for the following. In general, we ``loose'' one dimension because the projection $X\to\Spec\Z_{p\Z}$ cannot be deformed.

Here is the main idea of~\cite{FedorovMixedChar}: let $\bG$ be quasi-split with a Borel subgroup scheme $\bB$. Assume that $\dim X\ge2$. A generic trivialization of $\cE$ induces a generic reduction to $\bB$. Since $\bG/\bB$ is projective, this reduction can be extended to the complement of a subset of codimension two; call it $Z$. Now we recover the lost dimension: there is a smooth morphism $X'\to S$ similar to the above such that $Z\cap X'$ is $S$-finite (for example, if $\dim X=2$, then $Z$ is a scheme with finite underlying set, so it is automatically finite over $S$). One shows that this can be performed in such a way that there is a closed subscheme $Y\subset X'$, also $S$-finite, such that $Y\supset Z\cap X'$ and $X'-Y$ is affine. One then shows that $\cE$ can be reduced to the unipotent radical $\Ru\bB$ of $\bB$ on $X'-Y$, which shows that $\cE_{X'-Y}$ is trivial (because $X'-Y$ is affine). The rest of the proof is very similar to the equal characteristic case.

To make the above ideas work, it is required in~\cite{FedorovMixedChar} that $X$ has a projective compactification satisfying some technical conditions. It is also required that $\bG$ is split and only the local (rather than the semilocal) case is considered.

\subsubsection{} In~2020 \v{C}esnavi\v{c}ius (see~\cite{CesnaviciusGrSerre}) was able to get rid of the assumptions on the compactification, as well as to generalize to quasi-split group schemes and to the semilocal rings. Also, the proof was streamlined in loc.~cit., so we will generally follow it in this paper. One of the main ideas is that one need not choose $Y\subset X'$ as above at all. In fact, \v{C}esnavi\v{c}ius descends $\cE'$ to $\A^1_R$ by showing that $\Ru\bB$-torsors can be descended along certain affine morphisms.

\subsubsection{} We now briefly explain the argument in the paper.

\noindstep The first main idea of the current paper is to introduce a notion of a \emph{unipotent chain}, see Section~\ref{sect:Chains}. Roughly speaking, this is a sequence of torsors such that every torsor is obtained from the previous one via a unipotent modification. We prove (see Theorem~\ref{th:Chain}) that after reductions as in  Step~\eqref{step:start} of Section~\ref{sect:Overview}, one can find a closed subset $T$ of $X$ of codimension two (where we again assume that $\dim X\ge2$) such that $\cE_{X-T}$ is connected to the trivial torsor via a unipotent chain, provided that $\bG$ is simple, simply-connected, and isotropic. In this situation,~$\cE$ is trivial over every open affine subscheme of $X$ disjoint from $T$; see Theorem~\ref{th:AlmostTrivial} above and its proof.

Unfortunately, the above property of $\cE$ and $T$ is not enough to perform the descent as in Step~\eqref{step:Descend} of Section~\ref{sect:Overview}. To make this step work, we introduce a notion of a trivial unipotent chain with respect to a pair of opposite parabolic subgroup schemes of $\bG$; the above unipotent chain over $X-T$ is trivial away from a closed subset that is of positive codimension fiberwise over the Dedekind domain.

\noindstep Next, we prove the unramified case of the Grothendieck--Serre conjecture for any torsor that can be connected to a trivial torsor via a unipotent chain away from a subset of codimension at least two (there is also a local triviality condition for the chain). We follow the usual strategy. As in Step~\eqref{step:Descend} of Section~\ref{sect:Overview}, following essentially~\cite{FedorovMixedChar} and~\cite{CesnaviciusGrSerre}, one finds a smooth curve $C\to\Spec R$ with a section $\Delta$, a torsor $\cE'$ over~$C$ such that $\Delta^*\cE'\approx\cE$, and an $R$-finite closed subset $Z\subset C$ such that $\cE'_{C-Z}$ is connected to the trivial torsor via a unipotent chain. Moreover, we may choose the data so that the unipotent chain is trivial away from a closed subscheme quasi-finite over $R$.

\noindstep Initially, $\cE'$ is a torsor over a $C$-group scheme. However, replacing $C$ with an \'etale cover we reduce to the case when the group scheme is a pullback from $\Spec R$. This step, known as \emph{equating group schemes\/}, requires some additional work in our case (see Section~\ref{sect:Equating}). Since the notion of trivial unipotent chain depends on a pair of opposite parabolic subgroup schemes of $\bG$, we also need to make sure that this equating is compatible with parabolic subgroup schemes.

\noindstep Let the unipotent chain connecting $\cE'_{C-Z}$ to the trivial torsor be trivial away from a closed $R$-quasi-finite subset $Y\subset C$ containing $Z$. After improving our data, we get an \'etale morphism $\phi\colon C\to W\subset\A_R^1$ such that $Y$ maps isomorphically onto a closed subscheme $\phi(Y)\subset W$ such that the preimage of $\phi(Y)$ in $C$ is equal to $Y$. We descend $\cE'$ to a torsor $\cE_W$ over $W$. The triviality of the chain over $C-Y$ lets us descend the chain to a chain over $C-\phi(Z)$. In particular, $\cE_W$ is trivial away from every closed subscheme $Y'$ containing $\phi(Z)$ and having affine complement. Such a subscheme can be chosen to be $R$-finite. By patching $\cE_W$ with the trivial torsor over $\A^1_R-Y'$, we obtain a torsor over $\A^1_R$ trivial away from an $R$-finite closed subscheme.

\noindstep The last step is to generalize Step~\eqref{step:A1} of Section~\ref{sect:Overview} to the mixed characteristic case. This is accomplished by the following theorem of independent interest. We note that this theorem is closely related to the results of~\cite{PaninStavrovaA1}.

\begin{theorem}\label{th:A1}
Let $U$ be an affine semilocal scheme. Let~$\bG$ be a reductive group scheme over $U$. Assume that $Z$ is a closed subscheme of $\A^1_U$ finite over~$U$. Let~$\cE$ be a $\bG$-torsor over $\A^1_U$ trivial over $\A^1_U-Z$. Then for every section $\Delta\colon U\to\A^1_U$ of the projection $\A^1_U\to U$ the $\bG$-torsor $\Delta^*\cE$ is trivial.
\end{theorem}

This theorem will be proved in Section~\ref{sect:TorsorsA1}. If $U$ is a scheme over a field, then this is a slight generalization of~\cite[Thm.~4]{FedorovGrSerreNonSC}. The general case needs only minor modifications because most of the work is happening over the closed points of $U$ anyways.

\subsubsection{} We would like to mention some recent results concerning the mixed characteristic case of the Grothendieck--Serre conjecture: in~\cite{PaninMixedSL1} the conjecture is proved for $\SL_1(D)$, where $D$ is an Azumaya algebra over an unramified regular local ring. In~\cite{GuoPaninGrSerre} the conjecture is proved when $R$ is a semilocal ring of a scheme that is smooth and projective over a discrete valuation ring over which the group scheme is defined. More generally, in~\cite{GuoLiuGrSerre} the conjecture is proved in the unramified case, provided that $\bG$ is a constant group scheme.

\subsection{Notations} For a reductive group scheme $\bG$ we denote by $\bG^{\ad}$ its adjoint group scheme, and by $\Ru\bG$ its unipotent radical. By a simple group scheme we always mean a simple semisimple group scheme as in~\cite[Exp.~XXIV, 5.3]{SGA3-3}.

In this paper we work with right torsors; we only consider torsors for flat and finitely presented group schemes. If $\bG$ is a $T$-group scheme and $T'\to T$ be a morphism, we say ``a $\bG$-torsor over $T'$'' to mean a $\bG_{T'}$-torsor. We say that a $\bG$-torsor over $T'$ is \emph{generically trivial\/} if its restriction to some dense open subscheme of $T'$ is trivial.

Let $\bG$ be a $T$-group scheme and $T'\to T$ is a morphism. Let $\bH$ be a $T$-subgroup scheme of $\bG$ and $\cE$ be a $\bG$-torsor over $T'$. By an \emph{$\bH$-reduction\/} of $\cE$ we mean a pair $(\cH,\phi)$, where $\cH$ is an $\bH$-torsor over $T'$ and $\phi\colon\cH\times^\bH\bG\to\cE$ is an isomorphism of $\bG$-torsors. If such a reduction exists, then we say that $\cE$ \emph{can be reduced to $\bH$.}

\subsection{Acknowledgements} The author is grateful to Vladimir Chernousov, Philippe Gille, Ivan Panin, Raman Parimala, and Anastasia Stavrova for constant interest in his work. The author thanks K\k{e}stutis \v{C}esnavi\v{c}ius and Fei Liu for finding errors in early drafts. The author also thanks Dr.~\v{C}esnavi\v{c}ius for invaluable discussions. The author is grateful to anonymous referees for valuable suggestions.

A part of this work was done during the Georgia Algebraic Geometry Symposium at Emory University, the author wants to thank the organizers. The author is partially supported by the NSF DMS grants 2001516 and~2402553.

\section{Unipotent chains of torsors}\label{sect:Chains}
In this section, we define unipotent chains of torsors. The idea is that every torsor in the chain is obtained from the previous by ``modifying along a unipotent subgroup scheme.''

Let $\bG$ be a reductive group scheme and let $\cE$, $\cE'$ be $\bG$-torsors. The scheme of $\bG$-automorphisms of $\cE$ is a group scheme $\Aut(\cE)$; this is a reductive group scheme called \emph{a strongly inner form of $\bG$} (since $\cE$ is \'etale locally trivial, $\Aut(\cE)$ is isomorphic to $\bG$ \'etale locally). The scheme of $\bG$-isomorphisms $\Iso(\cE,\cE')$ is naturally an $\Aut(\cE)$-torsor. It is also a left $\Aut(\cE')$-torsor, though we can identify left and right torsors using the inversion in $\bG$. Recall that $\Ru$ stands for the unipotent radical.
\begin{definition}\label{def:UnipChain}
Let $X$ be a scheme; $\bG$ be a reductive $X$-group scheme. \emph{A unipotent chain of $\bG$-torsors} is a sequence
\begin{equation}\label{eq:chain}
\cE_1, \ldots, \cE_n, \bP_1, \ldots, \bP_{n-1}, \tau_1, \ldots, \tau_{n-1},
\end{equation}
where $\cE_i$ is a $\bG$-torsor, $\cE_1=\bG$ is a trivial $\bG$-torsor, $\bP_i\subset\Aut(\cE_i)$ is a parabolic subgroup scheme, and $\tau_i$ is a reduction of the $\Aut(\cE_i)$-torsor $\Iso(\cE_i,\cE_{i+1})$ to $\Ru\bP_i$. We also say that the chain \emph{connects} the trivial torsor to $\cE_n$.
\end{definition}
Note an obvious lemma.
\begin{lemma}\label{lm:ChainPullback}
  In the notation of the definition, let $f\colon Y\to X$ be a morphisms of schemes. Then $f^*\cE_1$, \ldots, $f^*\cE_n$, $f^*\bP_1$, \ldots, $f^*\bP_{n-1}$, $f^*\tau_1$, \ldots, $f^*\tau_{n-1}$ is a unipotent chain of $f^*\bG$-torsors.
\end{lemma}
\begin{lemma}\label{lm:ChainTrivAff}
    In the notation of the definition, assume that $X$ is affine. Then all the torsors $\cE_i$ are trivial.
\end{lemma}
\begin{proof}
  Induction on $n$ reduces the lemma to showing that $\cE_{n-1}\approx\cE_n$. Since $X$ is affine, every $\Ru\bP_{n-1}$-torsor is trivial by~\cite[exp.~XXVI, cor.~2.2]{SGA3-3}. Thus, $\Iso(\cE_{n-1},\cE_n)$ is a trivial $\Aut(\cE_{n-1})$-torsor. Hence, it has a section, which is the sought-for isomorphism.
\end{proof}

We note that the trivialized $\bG$-torsor has a canonical reduction to any subgroup scheme of $\bG$; we call this reduction \emph{standard}.
\begin{definition}\label{def:TrivUnipChain}
Let $X$ be a scheme; let $\bG$ be a reductive $X$-group scheme with fixed opposite parabolic subgroup schemes $\bP_\pm$. We say that a unipotent chain~\eqref{eq:chain} is \emph{trivial\/} if $\bP_1=\bP_\pm$, $\bG=\cE_1$, and for $i=2,\ldots,n$ there are trivializations of $\cE_i$ identifying $\bP_i$ either with $\bP_-$ or with $\bP_+$ and such that under these trivializations $\tau_i$ become the standard reductions of $\Iso(\cE_{i-1},\cE_i)=\bG$ to $\Ru\bP_{i-1}=\Ru\bP_\pm$.
\end{definition}
Note that the notion of a unipotent chain does not depend on a choice of opposite parabolic subgroup schemes $\bP_\pm$ but the notion of a trivial chain does. The pullback of a trivial chain is trivial. Theorem~\ref{th:Main} is a simple corollary of the following two theorems, Popescu's approximation theorem, and the Dedekind (that is, one-dimensional) case.
\begin{theorem}\label{th:Chain}
  Let $R$ be a Noetherian normal domain, let $\bG$ be a simply-connected simple $R$-group scheme with opposite proper parabolic subgroup schemes $\bP_\pm\subset\bG$, and let $\cE$ be a $\bG$-torsor over an open subscheme $U\subset\Spec R$. Assume that $\cE$ is trivial away from $V(f)\cap U$, where $f\in R-\{0\}$. Then there is a closed subset $T\subset V(f)$ of codimension at least two in $\Spec R$ and a unipotent chain of $\bG$-torsors over $\Spec R-T$:
  \[
    \cE_1=\bG_{\Spec R-T}, \cE_2, \ldots, \cE_n, \bP_1, \ldots, \bP_{n-1},\tau_1,\ldots,\tau_{n-1}
  \]
  such that $(\cE_n)_{U-T}\simeq\cE_{U-T}$ and the chain is trivial over $\Spec R-V(f)$.
\end{theorem}
We will give a proof in Section~\ref{sect:ProofOfChain}.

\begin{theorem}\label{th:Main2}
Let $X$ be an integral affine scheme smooth and of positive relative dimension over a semilocal Dedekind domain. Let $\bG$ be a reductive $X$-group scheme with opposite parabolic $X$-subgroup schemes $\bP_\pm$. Let $T\subset Y\subset X$ be closed subsets such that $Y$ is fiberwise of positive codimension over the Dedekind domain and $T$ is of codimension at least two in $X$. Let $\cE$ be a $\bG$-torsor over $X$ such that there is a unipotent chain of $\bG$-torsors over $X-T$ connecting $\bG_{X-T}$ to $\cE_{X-T}$ and assume that the chain is trivial over $X-Y$. Then $\cE$ is trivial Zariski semilocally over $X$ (that is, for any finite subset of $X$, the torsor $\cE$ is trivial in a Zariski neighborhood of this subset).
\end{theorem}
This theorem will be proved in Section~\ref{Sect:Proof}.

\begin{proof}[Derivation of Theorem~\ref{th:Main} from Theorems~\ref{th:Chain} and~\ref{th:Main2}]
Since $R$ is regular, $\Spec R$ is the disjoint union of its irreducible components, so we may assume that $R$ is integral.
Assume first that $\bG$ is a simple group scheme. By Popescu's Theorem (\cite[Tag07GC]{StacksProject}), we may assume that $R$ is a semilocal ring of a finite set $\bx$ of points on an integral affine scheme $X$ that is smooth over a semilocal Dedekind domain. By~\cite[Thm.~1]{Guo2019GrSerreDedekind}, we may assume that $X$ is of positive relative dimension over this Dedekind domain. Spreading out and replacing $X$ with an appropriate affine neighborhood of $\bx$, we may assume that $\bG$ and $\cE$ are defined over $X$ and that $\bG$ has a proper parabolic subgroup scheme $\bP_+$. Let $\bP_-$ be an opposite parabolic subgroup scheme (it exists because $X$ is affine, see~\cite[exp.~XXVI, cor.~4.3.5(i)]{SGA3-3}).

Applying again~\cite{Guo2019GrSerreDedekind} to the semilocalization of $X$ at the generic points of special fibers and spreading out, we find a closed subscheme $Y\subset X$ fiberwise of positive codimension over the Dedekind domain such that $\cE$ is trivial away from $Y$. Moreover, we may assume that $Y=V(f)$ for $f\in\Gamma(X,\cO_X)$.

Applying Theorem~\ref{th:Chain} with $R=\Gamma(X,\cO_X)$, $U=X$, we get a closed subset $T\subset Y$ of codimension at least two in $X$ and a unipotent chain of $\bG$-torsors over $X-T$ connecting $\bG_{X-T}$ to $\cE_{X-T}$. Moreover, the chain is trivial over $X-Y$. It remains to apply Theorem~\ref{th:Main2}.

Consider now the general case. Recall that a semisimple simply-connected group scheme can be written as the product of Weil restrictions of simple group schemes along finite connected \'etale covers (see~\cite[exp.~XXIV, prop.~5.10]{SGA3-3}). Write $\bG=\prod_{i=1}^r\bG_i$, where $\bG_i$ is the Weil restriction of a simple $R_i$-group scheme $\overline\bG_i$, where the integral domain $R_i$ is a finite \'etale $R$-algebra. In particular, each $R_i$ is semilocal. We need a lemma.
\begin{lemma}\label{lm:TotallyIsotropic}
  The simply-connected group scheme $\bG$ has a strictly proper parabolic subgroup scheme if and only if each $\overline\bG_i$ has a proper parabolic subgroup scheme.
\end{lemma}
\begin{proof}
  Let $\cP$ be the scheme of parabolic subgroup schemes of $\bG$, let $\cP_i$ be the scheme of parabolic subgroup schemes of $\overline\bG_i$ (cf.~\cite[exp.~XXVI, cor.~3.5]{SGA3-3}). Since a parabolic subgroup scheme of a reductive group scheme gives rise to a parabolic subgroup scheme of its Weil restriction, we get a morphism $\prod_i\Res_{R_i/R}\cP_i\to\cP$, where $\Res_{R_i/R}$ is the Weil restriction functor. Let us check that this morphism is an isomorphism. Since this is enough to check after an \'etale base change, we may assume that for all $i$ we have $R_i=R$ and that each $\overline\bG_i$ is split. Then the statement reduces to the statement that the parabolic subgroup schemes of $\prod_{i=1}^r\overline\bG_i$ are exactly subgroup schemes of the form $\bP_1\times\ldots\times\bP_r$, where each $\bP_i$ is a parabolic subgroup scheme of $\overline\bG_i$. Choosing a split maximal torus and a Borel subgroup scheme in $\overline\bG_i$, we get a split maximal torus and a Borel subgroup scheme in $\prod_i\overline\bG_i$. Then we have a notion of standard parabolic subgroup schemes, and it is clear that every standard parabolic subgroup scheme in $\prod_i\overline\bG_i$ is the product of standard parabolic subgroup schemes in $\overline\bG_i$. It remains to note that every parabolic subgroup scheme is locally conjugate to a standard one.

  We see that parabolic subgroup schemes of $\bG$ correspond to collections $(\bP_1,\ldots,\bP_r)$, where each $\bP_i\subset\overline\bG_i$ is a parabolic subgroup scheme. It is clear that strictly proper parabolic subgroup schemes correspond to collections of proper parabolic subgroup schemes.
\end{proof}
By~\cite[exp.~XXIV, prop.~8.4]{SGA3-3}, $\cE$ corresponds to a sequence $\cE_1$, \ldots, $\cE_r$ of $\overline\bG_i$-torsors. These torsors are generically trivial. By the previous lemma, each $\overline\bG_i$ contains a proper parabolic subgroup scheme. Thus, by the already settled case, $\cE_i$ is trivial, so $\cE$ is trivial as well. This completes the proof of Theorem~\ref{th:Main}.
\end{proof}

\begin{proof}[Derivation of Theorem~\ref{th:AlmostTrivial} from Theorem~\ref{th:Chain}]
    As in the proof above, using Lemma~\ref{lm:TotallyIsotropic}, we may assume that $\bG$ is simple, simply-connected, and isotropic. By Theorem~\ref{th:Chain}, there is a closed subset $T\subset\Spec R$ of codimension at least two and a unipotent chain of $\bG$-torsors over $\Spec R-T$:
  \[
    \cE_1=\bG_{\Spec R-T}, \cE_2, \ldots, \cE_n, \bP_1, \ldots, \bP_{n-1},\tau_1,\ldots,\tau_{n-1}
  \]
  such that $(\cE_n)_{U-T}\simeq\cE_{U-T}$. Let $\Spec A\to U-T$ be a morphism. By Lemma~\ref{lm:ChainPullback} we get a unipotent chain $\bG_A$, $(\cE_2)_A$, \ldots, $\cE_A$, $(\bP_1)_A$, \ldots. By Lemma~\ref{lm:ChainTrivAff}, $\cE_A$ is trivial.
\end{proof}

\begin{remarks}\label{rem:rem2}
\noindstep\label{rem:Ces} We can easily derive~\cite[Thm.~9.1]{CesnaviciusGrSerre} from Theorem~\ref{th:Main2}. We argue as in the proof of~\cite[Prop.~4.2]{CesnaviciusGrSerre}. By Popescu's Theorem we may assume that $\bG$ is a quasi-split reductive $X$-group scheme, where $X$ is as in the proof of Theorem~\ref{th:Main}. Let $\bB\subset\bG$ be a Borel subgroup scheme and $\cE$ be a generically trivial $\bG$-torsor. Similarly to the proof of Theorem~\ref{th:Main}, we can find $Y\subset X$ that is fiberwise over the Dedekind domain of positive codimension and such that $\cE$ is trivial away from~$Y$; we fix a trivialization. Since the scheme $\cE/\bB$ classifying $\bB$-reductions of $\cE$ is $R$-projective, the standard reduction of the trivial $\bG$-torsor $\cE_{X-Y}$ to $\bB$ can be extended to a closed subset $T\subset Y$ of codimension at least two and we denote the $\bB$-reduction thus obtained by $(\cB,\phi)$ (that is, $\cB$ is a $\bB$-torsor over $X-T$ and $\phi\colon\cB\times^\bB\bG\to\cE_{X-T}$ is an isomorphism). Consider the torus $\bT:=\bB/\Ru\bB$. Then $\cB/\Ru\bB$ is a $\bT$-torsor, trivialized over $X-Y$. By~\cite[cor.~6.9]{ColliotTS1979fibres} this torsor extends to~$X$, so by the Grothendieck--Serre conjecture for tori (see~\cite[thm.~4.1(i)]{ColliotTheleneSansuc}) it is Zariski semilocally trivial. Thus, after multiplying the trivialization of $\cE$ over $X-Y$ by a section of $\bT$ over $X-Y$, the standard $\Ru\bB$ reduction of $\cE_{X-Y}$ extends to a~reduction $\tau$ of~$\cE$ to $\Ru\bB$ over $X-T$. Now $\bG_{X-T}$, $\cE_{X-T}$, $\bB_{X-T}$, $\tau$ is a unipotent chain connecting the trivial torsor to $\cE_{X-T}$. Moreover, this chain is trivial over $X-Y$. Now we can apply Theorem~\ref{th:Main2}.

We note that our argument is somewhat simpler than the original one as we do not have to descend torsors under unipotent group schemes along affine morphisms as in~\cite[Sect.~7]{CesnaviciusGrSerre}. This simplification is made possible by Proposition~\ref{pr:DescendChain} below.

\noindstep It seems plausible, that for Theorem~\ref{th:Chain} we only need $\bG$ to have a proper parabolic subgroup scheme generically. Unfortunately, we cannot prove the Grothendieck--Serre conjecture in the case, when $\bG$ is only generically totally isotropic: the problem is that we cannot equate generic parabolic subgroup schemes as in Proposition~\ref{pr:EquatingGroups}.
\end{remarks}

\section{Unipotent chains: Proof of Theorem~\ref{th:Chain}}\label{sect:ProofOfChain}

\subsection{Gluing and reductions of torsors}\label{sect:Gluing}
Let $X$ be a scheme and fix a Zariski cover $X=W_1\cup W_2$. Let $\bH$ be a flat and finitely presented $X$-group scheme and let $g\in\bH(W_1\cap W_2)$. Then we can use $g$ to glue the trivial $\bH$-torsor over $W_1$ with the trivial $\bH$-torsor over $W_2$. We obtain an $\bH$-torsor $\Glue(\bH,g)$ over $X$. More precisely, the torsor $\Glue(\bH,g)$ has canonical trivializations $s_1$ over $W_1$ and $s_2$ over $W_2$ such that $s_2|_{W_1\cap W_2}=s_1|_{W_1\cap W_2}g$. Being trivialized over $W_1$, the $\bH$-torsor $\Glue(\bH,g)$ has a canonical reduction to any subgroup scheme of $\bH$ over $W_1$. We recall that such a reduction is called standard.

\begin{lemma}\label{lm:GlueRed}
  Notations being as above, assume that $\bU\subset\bH$ is an $X$-subgroup scheme and let $u\in\bU(W_1\cap W_2)$.

  (i) The $\bH$-torsor $\Glue(\bH,u)$ has a $\bU$-reduction that extends the standard reduction over $W_1$.

  (ii) Assume that $g\in\bH(W_1\cap W_2)$ and set $\cF:=\Glue(\bH,g)$, $\cF':=\Glue(\bH,ug)$. We use the trivialization of $\cF$ over $W_1$ to identify the restrictions of $\bH$ and $\Aut(\cF)$ to $W_1$. Assume that there is an $X$-subgroup scheme $\bU'\subset\Aut(\cF)$ such that under the above identification we have $\bU_{W_1}=\bU'_{W_1}$. Then the $\Aut(\cF)$-torsor $\Iso(\cF,\cF')$ has a $\bU'$-reduction that extends the standard reduction over $W_1$.
\end{lemma}
\begin{proof}
  (i) Consider the $\bU$-torsor $\Glue(\bU,u)$. It follows from the definitions that $\Glue(\bU,u)\times^\bU\bH\simeq\Glue(\bH,u)$.

  (ii) Let $s_i$ be the canonical trivialization of $\cF$ over $W_i$, while $s'_i$ be the canonical trivialization of $\cF'$ over $W_i$. These trivializations give sections (=trivializations) $t_i$ of $\Iso(\cF,\cF')$ over $W_i$ sending $s_i$ to $s'_i$. Then, on $W_1\cap W_2$, we have $t_2(s_1)=t_1(s_1)u$.

  Next, on $W_1\cap W_2$, we use the trivializations $s_1$ and $t_1$ to identify $\Aut(\cF)$ and $\Iso(\cF,\cF')$ with $\bG$. Then the action of $\Aut(\cF)$ on $\Iso(\cF,\cF')$ is identified with the action of $\bG$ on itself via the multiplication on the right. Thus we have on $W_1\cap W_2$: $t_2=t_1u$. Therefore, $\Iso(\cF,\cF')=\Glue(\Aut(\cF),u)$. It remains to apply part (i) with $\bU'\subset\Aut(\cF)$ instead of $\bU\subset\bH$.
\end{proof}

\subsection{Trivializing away from two divisors} The main result of this section is Proposition~\ref{pr:Away2Div}.
\begin{lemma}\label{lm:Approximation}
  Let $R$ be a Noetherian normal domain and $f\in R-\{0\}$. Let $\fp_1,\ldots,\fp_m\subset R$ be all the distinct minimal prime ideals containing $f$. Let $\overline\cO_j$ be the completion of the discrete valuation ring $R_{\fp_j}$ and let $\overline K_j$ be the fraction field of $\overline\cO_j$. Let $M$ be a finitely generated $R$-module. Assume that for all $j$ we are given $u_j\in M\otimes_R\overline K_j$. Then for all $N>0$ there are $h\in R$ and $u\in M\otimes_R R_{fh}$ such that for all $j$ we have $h\notin\fp_j$ and $u\equiv u_j\pmod{M\otimes_R\fp_j^N\overline\cO_j}$, where we view $u$ as an element of $M\otimes_R\overline K_j$ via the composition $R_{fh}\to\Frac(R)=\Frac(R_{\fp_j})\to\overline K_j$.
\end{lemma}
\begin{proof}
   For $l\gg0$ we have $f^lu_j\in M\otimes_R\overline\cO_j$ for all $j$. Consider the semilocal Dedekind domain $R_{\fp_1,\ldots,\fp_m}$ and the module $M_{\fp_1,\ldots,\fp_m}:=M\otimes_RR_{\fp_1,\ldots,\fp_m}$. Then by the Chinese Remainder Theorem we can find $v\in M_{\fp_1,\ldots,\fp_m}$ such that for all $j$ we have $f^lu_j\equiv v\pmod{M\otimes_R\fp_j^{N+k_jl}\overline\cO_j}$, where $k_j$ is the valuation of $f$ in $R_{\fp_j}$. Write $v=\frac xh$, where $x\in M$ and $h\notin\fp_j$ for all $j$. It remains to take $u:=\frac x{f^lh}$.
\end{proof}

We will keep the notation $R$, $f$, $\fp_j$, $\overline\cO_j$ and $\overline K_j$ through the end of the section. For an $R$-group scheme $\bH$, we denote by $\bH^{(N)}_j$ the $N$-th congruence subgroup of $\bH(\overline\cO_j)$, that is, the kernel of the homomorphism $\bH(\overline\cO_j)\to\bH(\overline\cO_j/\fp_j^N\overline\cO_j)$.
\begin{lemma}\label{lm:ApproxU}
  Let $R$, $f$, $\fp_j$, $\overline\cO_j$ and $\overline K_j$ be as above, let $\bU$ be an affine $R$-group scheme with a filtration $\bU=\bU_n\supset\bU_{n-1}\supset\ldots\supset\bU_0=\{e\}$ such that the successive quotients $\bU_i/\bU_{i-1}$ are isomorphic to additive group schemes of locally free $R$-modules of finite rank. Let $u_j\in\bU(\overline K_j)$ for $j=1,\ldots,m$. Then for $N>0$ there are $h\in R$ and $u\in\bU(R_{fh})$ such that for all $j$ we have $h\notin\fp_j$ and $u_j\in \bU^{(N)}_ju$.
\end{lemma}
\begin{proof}
  Let $l$ be an integer such that $0\le l\le n$. Assume that for all $j$ we have $u_j\in\bU_l(\overline K_j)$. We will show that for all $N>0$ we can find $h\in R$ and $u\in\bU_l(R_{fh})$ such that for all $j$ we have $h\notin\fp_j$ and $u_j\in (\bU_l)^{(N)}_ju$. When $l=n$, this statement is the claim of our lemma.

  We induct on $l$, the case $l=0$ being obvious. Assume that the statement is known for $l-1$. Let $\bar u_j\in(\bU_l/\bU_{l-1})(\overline K_j)$ be the image of $u_j$ in the quotient. By Lemma~\ref{lm:Approximation} we can find $h'\in R$ and $\bar w\in(\bU_l/\bU_{l-1})(R_{fh'})$ such that for all $j$ we have $h'\notin\fp_j$ and $\bar u_j=\bar v_j \bar w$, where $\bar v_j\in (\bU_l/\bU_{l-1})^{(N)}_j$.

  Since $H^1(R_{fh'},\bU_{l-1})=0$, we can lift $\bar w$ to an element $w\in\bU_l(R_{fh'})$. We claim that we can lift $\bar v_j$ to an element $v_j\in(\bU_l)^{(N)}_j$. Indeed, since $H^1(\overline\cO_j,\bU_{l-1})=0$, we can lift $\bar v_j$ to an element $v'\in\bU_l(\overline\cO_j)$. The image of $v'$ in $\bU_l(\overline\cO_j/\fp_j^N\overline\cO_j)$ is contained in $\bU_{l-1}(\overline\cO_j/\fp_j^N\overline\cO_j)$.
  Let $v''$ be a lift of this element to $\bU_{l-1}(\overline\cO_j)$. We can take $v_j:=v'(v'')^{-1}$.

  Then for all $j$ we have $u_j=v_jwu'_j$, where $u'_j\in\bU_{l-1}(\overline K_j)$. Since conjugation by any element is a continuous automorphism of $\bU_l(\overline K_j)$, we can find $M>0$ such that for all $j$ and all $a\in(\bU_l)^{(M)}_j$ we have $waw^{-1}\in(\bU_l)^{(N)}_j$. Applying the induction hypothesis, we find $h''\in R$ and $u'\in\bU_{l-1}(R_{fh''})$ such that for all $j$ we have $h''\notin\fp_j$ and $u'_j=a_ju'$, where $a_j\in(\bU_{l-1})^{(M)}_j$.

  Finally, take $h=h'h''$, we have
  \[
    u_j=v_jwu'_j=v_jwa_ju'=v_j(wa_jw^{-1})wu'\in(\bU_l)^{(N)}_j(wu').
  \qedhere\]
\end{proof}

\begin{lemma}\label{lm:ApproxG}
  Let $R$, $f$, $\fp_j$, $\overline\cO_j$ and $\overline K_j$ be as above, let $\bG$ be a simply-connected simple $R$-group scheme with proper opposite parabolic subgroup schemes $\bP_\pm$. Let $g_j\in\bG(\overline K_j)$ for $j=1,\ldots,m$. Then for some $n\ge0$ there are $h\in R$ and $u_1,\ldots,u_n\in \Ru\bP_\pm(R_{fh})$ such that for all $j$ we have $h\notin\fp_j$ and $g_j\in\bG(\overline\cO_j) u_1\ldots u_n$.
\end{lemma}
\begin{proof}
For $n\gg0$ and each $j=1,\ldots,m$ using~\cite[fait~4.3(2), lemme~4.5(1)]{Gille:BourbakiTalk} we find elements $u^j_i\in \Ru\bP_\pm(\overline K_j)$ ($i=1,\ldots,n$) such that for all $j$ we have $g_j\in\bG(\overline\cO_j) u^j_1\ldots u^j_n$. Moreover, inserting 1's into the collection $u^j_i$ we may assume that for each $i$ either $u^1_i,\ldots,u^m_i\in\Ru\bP_+(\overline K_j)$, or $u^1_i,\ldots,u^m_i\in\Ru\bP_-(\overline K_j)$.

It is enough to show that for $0\le l\le n$ there are $h\in R$ and $u_i\in \Ru\bP_\pm(R_{fh})$, $i=1,\dots,l$ such that for all $j$ we have $h\notin\fp_j$ and $u^j_1\ldots u^j_l\in\bG(\overline\cO_j)u_1\ldots u_l$. We induct on $l$, the case $l=0$ being obvious.

Applying the induction hypothesis, we get $h'\in R$ and $u_i\in \Ru\bP_\pm(R_{fh'})$, $i=1,\dots,l-1$ such that for all $j$ we have $h'\notin\fp_j$ and $u^j_1\ldots u^j_{l-1}\in\bG(\overline\cO_j)u_1\ldots u_{l-1}$. As in the proof of Lemma~\ref{lm:ApproxU}, there is an integer $N>0$ such that for all $j$ and all $g\in\bG_j^{(N)}$ we have: $(u^j_1\ldots u^j_{l-1})g(u^j_1\ldots u^j_{l-1})^{-1}\in\bG(\overline\cO_j)$. Applying Lemma~\ref{lm:ApproxU} (which is applicable by~\cite[exp.~XXVI, prop.~2.1]{SGA3-3}) to $u^j_l$, we get $h''\in R$ and $u_l\in \Ru\bP_\pm(R_{fh''})$ such that for all $j$ we have $h''\notin\fp_j$ and $u^j_l=g_ju_l$ with $g_j\in\bG_j^{(N)}$. Take $h=h'h''$.

By our choice of $N$, we have
    \[
        u^j_1\ldots u^j_l=u^j_1\ldots u^j_{l-1}g_ju_l=\left((u^j_1\ldots u^j_{l-1})g_j(u^j_1\ldots u^j_{l-1})^{-1}\right)
        u^j_1\ldots u^j_{l-1}u_l\in\bG(\overline\cO_j)u_1\ldots u_l.
    \qedhere\]
\end{proof}

For $f,h\in R$, let $T:=V(f,h)$ be the closed subset of $\Spec R$ corresponding to the ideal $fR+hR$. Then we have a Zariski cover $\Spec R-T=\Spec R_f\cup\Spec R_h$. Thus, we are in the situation of Section~\ref{sect:Gluing} with $W_1=\Spec R_f$, $W_2=\Spec R_h$. Hence, for $g\in\bG(R_{fh})$ we have a $\bG$-torsor $\Glue(\bG,g)$ over $\Spec R-T$.

\begin{proposition}\label{pr:Away2Div}
  In the notation of Theorem~\ref{th:Chain}, let $\bG$ be a simply-connected simple $R$-group scheme. Then there are $h\in R$ such that $T:=V(f,h)$ is of codimension at least two in $\Spec R$ and $u_1$, \ldots, $u_n$ such that $u_i\in \Ru\bP_\pm(R_{fh})$ and $\cE_{U-T}\simeq\Glue(\bG,u_n\ldots u_1)_{U-T}$.
\end{proposition}
\begin{proof}
We fix a trivialization $s$ of $\cE$ over $U-V(f)$. Let $\fp_1$, \ldots $\fp_m$ be all the minimal prime ideals of $R$ containing $f$. Let, as above, $\overline\cO_j$ be the completion of $R_{\fp_j}$, and $\overline K_j:=\Frac(\overline\cO_j)$.

By reordering the ideals $\fp_i$, we may assume that $\fp_1,\ldots,\fp_l\in U$, $\fp_{l+1},\ldots,\fp_m\notin U$ for a certain $l$ such that $0\le l\le m$. The semi-localization $R_{\fp_1,\ldots,\fp_l}$ is a Dedekind ring. By~\cite{Guo2019GrSerreDedekind} $\cE$ is trivial over it. We fix a trivialization $s'$. Now, both $s$ and $s'$ give trivializations of $\cE$ over $\overline K_j$, where $1\le j\le l$. Thus, on $\Spec\overline K_j$ we have $s'=sg_j$, where $g_j\in\bG(\overline K_j)$. We set $g_j=1$ for $j>l$. By applying Lemma~\ref{lm:ApproxG} to $g_j^{-1}$, we get $h\in R$ and $u_i\in\Ru\bP_\pm(R_{fh})$, $i=1,\ldots,n$ such that for all $j$ we have $h\notin\fp_j$ and $g_j\in u_n\ldots u_1\bG(\overline\cO_j)$. Set $T:=V(f,h)$ and note that $T$ has codimension at least two in $\Spec R$ because $h$ is not contained in any minimal prime ideal containing $f$. Consider the torsor $\cE':=\Glue(\bG,u_n\ldots u_1)$ over $\Spec R-T$. Since both $\cE$ and $\cE'$ are trivialized over $U-V(f)$, we get an isomorphism $\sigma\colon\cE_{U-V(f)}\xrightarrow{\sim}\cE'_{U-V(f)}$. We claim that $\sigma$ extends to $\fp_j$ for all $1\le j\le l$. Indeed, let us identify $\cE_{\overline\cO_j}$ and $\cE'_{\overline\cO_j}$ with $\bG_{\overline\cO_j}$ using $s'$ and the canonical trivialization respectively. Then $\sigma|_{\overline K_j}$ is identified with the multiplication by $(u_n\ldots u_1)^{-1}g_j$. In more detail, the section $s'|_{\overline K_j}$ of $\cE|_{\overline K_j}$ is sent by $\sigma|_{\overline K_j}$ to $g_j$, where we use the canonical trivialization of $\cE'$ on $U-V(f)$. Thus it is sent to $(u_n\ldots u_1)^{-1}g_j$ if we use the canonical trivialization of $\cE'$ on $\Spec\overline\cO_j$.

Since $(u_n\ldots u_1)^{-1}g_j\in\bG(\overline\cO_j)$, the claim that $\sigma$ extends to $\fp_j$ for all $1\le j\le l$ follows. Now, by Hartogs principle, the isomorphism extends to $U-T$.
\end{proof}

\subsection{Proof of Theorem~\ref{th:Chain}}
Recall that $\bP_\pm\subset\bG$ are opposite proper parabolic subgroup schemes. Let $h$, $T$, and $u_i$ be those provided by Proposition~\ref{pr:Away2Div}. Set $\cE_i:=\Glue(\bG,u_{i-1}\ldots u_1)$. We have $u_i\in\Ru\bP_-(R_{fh})$ or $u_i\in \Ru\bP_+(R_{fh})$. Consider the first case, the second one being completely similar. We note that $\cE_i$ is trivialized over $\Spec R_f$, which gives an isomorphism $\bG_{R_f}\xrightarrow{\sim}\Aut(\cE_i)_{R_f}$. We use this isomorphism to view $(\bP_-)_{R_f}$ as a parabolic $R_f$-subgroup scheme $\bP_i'\subset\Aut(\cE_i)_{R_f}$. Since $\Spec R-T$ is normal and the scheme classifying parabolic subgroup schemes of $\Aut(\cE_i)$ is proper over $\Spec R-T$, $\bP_i'$ extends to a parabolic subgroup scheme $\bP_i\subset\Aut(\cE_i)_{\Spec R-T_i}$, where $T\subset T_i\subset V(f)$ and $T_i$ has codimension at least two in $\Spec R$.

By Lemma~\ref{lm:GlueRed}(ii) applied with $X=\Spec R-T_i$, $W_1:=\Spec R_f$, $W_2:=\Spec R_h-T_i$, $\bH:=\bG_{\Spec R-T_i}$, $\bU:=(\Ru\bP_-)_{R-T_i}$, $g=u_{i-1}\ldots u_1$, $u=u_i$, $\Iso(\cE_i,\cE_{i+1})$ is reduced to $\Ru\bP_i$ over $\Spec R-T_i$. By construction and Proposition~\ref{pr:Away2Div}, $(\cE_n)_{U-T}\simeq\cE_{U-T}$. It remains to replace $T$ with $\bigcup_i T_i$, $\cE_i$ with $(\cE_i)_{\Spec R-T}$, and $\bP_i$ with $(\bP_i)_{\Spec R-T}$. By construction and Lemma~\ref{lm:GlueRed}(ii), the chain thus obtained is trivial over $\Spec R_f$.
\qed

\section{Lifting to a relative curve and descent to \texorpdfstring{$\A^1$}{A1}}\label{sect:Prepare}
The main result of this section is Proposition~\ref{pr:A1}. Theorem~\ref{th:Main2} is a simple corollary of this proposition and of Theorem~\ref{th:A1}.

\subsection{Fibering into curves}
Recall that for a reductive group scheme $\bG$ we have the notion of a unipotent chain, see Definition~\ref{def:UnipChain}. If we choose a pair of opposite parabolic subgroup schemes $\bP_\pm\subset\bG$, we get the notion of a trivial unipotent chain, see Definition~\ref{def:TrivUnipChain}. The following proposition is similar to~\cite[Prop.~4.2]{CesnaviciusGrSerre} (cf.\ also~\cite[Prop.~4.4]{FedorovMixedChar}). The main difference is item~\eqref{item:ext}. The torus $\bT$ is only needed to equate the group schemes later in Section~\ref{sect:Equating}.

\begin{proposition}\label{pr:PrepDescend}
Let $X$ be an integral affine scheme smooth and of positive relative dimension over a semilocal Dedekind domain. Let $\bG$ be a reductive $X$-group scheme with opposite parabolic $X$-subgroup schemes $\bP_\pm\subset\bG$ and a maximal $X$-torus $\bT\subset\bP_-\cap\bP_+$. Let $T\subset Y\subset X$ be closed subsets such that $Y$ is fiberwise of positive codimension over the Dedekind domain and $T$ is of codimension at least two in $X$. Let $\cE$ be a $\bG$-torsor over $X$ such that there is a unipotent chain of $\bG$-torsors over $X-T$ connecting $\bG_{X-T}$ to $\cE_{X-T}$ and assume that the chain is trivial over $X-Y$. Then for a finite set $\bx$ of points of~$X$, there are

\stepzero\noindstep a smooth integral affine $\cO_{X,\bx}$-scheme $C$ of pure relative dimension one;

\noindstep a section $\Delta\in C(\cO_{X,\bx})$;

\noindstep an $\cO_{X,\bx}$-finite closed subscheme $Z\subset C$;

\noindstep\label{eq:d} a reductive $C$-group scheme $\bG'$ with opposite parabolic $C$-subgroup schemes $\bP'_\pm\subset\bG'$ and a maximal $C$-torus $\bT'\subset\bP'_-\cap\bP'_+$ such that the $\Delta$-pullback of the data $\bT',\bP'_\pm,\bG'$ is identified with the restriction of $\bT,\bP_\pm,\bG$ to $\cO_{X,\bx}$;

\noindstep a $\bG'$-torsor $\cE'$ such that $\Delta^*\cE'\approx\cE_{\cO_{X,\bx}}$, where the isomorphism makes sense in view of the identifications in part~\eqref{eq:d}; and

\noindstep\label{item:ext} a unipotent chain of $\bG'$-torsors over $C-Z$ connecting $\bG'_{C-Z}$ to $\cE'_{C-Z}$ that is trivial (with respect to $\bP'_\pm$) away from an $\cO_{X,\bx}$-quasi-finite subscheme of $C$.
\end{proposition}
\begin{proof}
The argument is similar to that of~\cite[Prop.~4.2]{CesnaviciusGrSerre}. In more detail, set $U:=\Spec\cO_{X,\bx}$. Denote the semilocal Dedekind domain by $\cO$. Applying~\cite[Prop.~4.1]{CesnaviciusGrSerre}, we obtain an affine open subscheme $X'\subset X$ containing $\bx$, an affine open subscheme $S\subset\A_{\cO}^{\dim X-2}$, and a smooth morphism $X'\to S$ of pure relative dimension one such that $T\cap X'$ is finite over $S$. Since $Y$ is fiberwise of positive codimension over the Dedekind domain, using a standard argument one can arrange that also $Y\cap X'$ is quasi-finite over $S$ (see, for example,~\cite[Prop.~2.3]{FedorovCesnavicius}).

We now let $C:=X'\times_SU$ and $Z:=(T\cap X')\times_SU$. We let $\Delta\colon U\to C$ be the diagonal section. Now we replace $C$ with its connected component containing $\Delta(U)$ and replace $Z$ by the intersection with this component. Let $\bG'$, $\bP'_\pm$, $\bT'$, and $\cE'$ be the pullbacks of $\bG$, $\bP_\pm$, $\bT$, and $\cE$ under the composition $C\to X'\hookrightarrow X$. Let the unipotent chain over $C-Z$ be the pullback of the unipotent chain over $X-T$ (see Lemma~\ref{lm:ChainPullback}). The chain is trivial away from the $\cO_{X,\bx}$-quasi-finite subscheme $(Y\cap X')\times_SU$. All the conditions of the proposition are now satisfied by construction.
\end{proof}

\subsection{Equating the group schemes}\label{sect:Equating} We will later show that in Proposition~\ref{pr:PrepDescend} we may assume that the group scheme $\bG'$ is the pullback of $\bG$ under the projection $C\to\Spec\cO_{X,\bx}$. To that end, we will use Proposition~\ref{pr:EquatingGroups} below. We start with some preliminaries.

\begin{lemma}\label{lm:FibDense}
  Let $R$ be an integral domain and let $\bT\subset\gmR^n$ be a subtorus. Let $\P_R^n$ be the standard compactification of $\gmR^n\subset\A_R^n$. Let $\overline\bT$ be the closure of $\bT$ in $\P_R^n$. Then $\bT$ is fiberwise dense in $\overline\bT$.
\end{lemma}
\begin{proof}
Consider the standard affine cover $\P_R^n=\bigcup_{i=0}^n U_i$, where $U_i\approx\A_R^n$. Note that $\bT\subset\gmR^n\subset U_i$. It is enough to show that $\bT$ is fiberwise dense in its closure in $U_i$ for all $i$. Denote this closure by $\overline\bT_i$.

Fix $i$ such that $0\le i\le n$. We identify the character lattice of $\gmR^n$ with $\Z^n$ in such a way that $U_i=\Spec R[\Z_{\ge0}^n]$. Let $\Lambda:=\Hom_R(\bT,\gmR)$ be the character lattice of $\bT$ (this is a free abelian group). The embedding $\bT\hookrightarrow\gmR^n$ corresponds to a surjective homomorphism $\pi\colon\Z^n\to\Lambda$. Let $X:=\Spec R[\pi(\Z_{\ge0}^n)]$. Then $X$ is a closed integral subscheme of $U_i$ containing $\bT$. Since $\pi(\Z_{\ge0}^n)$ generates $\Lambda$, we see that $\bT$ is open in $X$, so that $X=\overline\bT_i$.

Similarly, for any point $u\in\Spec A$, we see that $X_u=(\overline\bT_i)_u$ is the closure of $\bT_u$ in $\A^n_u$ so that $\bT_u$ is dense in $(\overline\bT_i)_u$.
\end{proof}

The following proposition about tori over normal semilocal schemes may be of independent interest.
\begin{proposition}\label{pr:TorTorsors}
Let $W$ be an affine integral normal Noetherian scheme and let $\bT$ be a~$W$-torus. Let $\cT$ be a $\bT$-torsor. Then

\stepzero\noindstep\label{item:torus1} $\bT$ has a $\bT$-equivariant fiberwise compactification, that is, there is a projective $W$-scheme $\overline\bT$ such that $\bT$ is a fiberwise dense open subscheme of $\overline\bT$ and there is an action of $\bT$ on $\overline\bT$ extending the action on itself by multiplication;

\noindstep\label{item:torus2} similarly, $\cT$ has a fiberwise compactification, that is, there is a projective $W$-scheme $\overline\cT$ such that~$\cT$ is a fiberwise dense open subscheme of $\overline\cT$;

\noindstep\label{item:torus3} if $W$ is semilocal, $U$ is a closed subscheme of $W$, and $\delta\colon U\to\cT$ is a section of $\cT\to W$, then there is a closed subscheme $\widetilde W\subset\cT$ finite and \'etale over $W$ and containing $\delta(U)$.
\end{proposition}
\begin{proof}
\eqref{item:torus1} Since $W$ is affine, Noetherian, and normal, by~\cite[Cor.~3.2(3)]{ThomasonResolution} we have an embedding $\bT\to\GL_{n,W}$ for some $n>0$. Put $N:=n^2$, and let $\P_W^N=\P(\gl_{n,W}\oplus\cO_W)$ be the standard compactification of $\GL_{n,W}$. Let $\overline\bT$ be the closure of $\bT$ in $\P_W^N$. We need to check that $\bT$ is fiberwise dense in $\overline\bT$. This is enough to check on a finite \'etale cover because the connected components of such a cover are integral by~\cite[Tag0BQL]{StacksProject} and normal by descent~\cite[Tag034F]{StacksProject}. Hence, by~\cite[Exp.~X, thm.~5.16]{SGA3-2} we may assume that $\bT$ is a diagonalizable (=split) torus. Thus, we may assume that $\bT$ is contained in the torus $\gmW^n$ of diagonal matrices and the statement follows from Lemma~\ref{lm:FibDense}.

\eqref{item:torus2} Notations as in the previous proof, note that $\bT$ acts on the vector bundle $\cV:=\gl_{n,W}\oplus\cO_W$, so we have a closed embedding
\[
    \overline\cT:=\cT\times^\bT\overline\bT\hookrightarrow\P(\cT\times^\bT\cV).
\]
Now $\overline\cT$ is the required compactification. (These statements can be checked \'etale locally, so we may assume that $\cT$ is a trivial $\bT$-torsor, in which case these are obvious.)

\eqref{item:torus3} From the previous part we obtain $\cT\subset\overline\cT\subset\P_W^N$. Let $w\in W$ be a closed point. Using Bertini's Theorem (see~\cite{PoonenOnBertini} and~\cite[exp.~XI, thm.~2.1(ii)]{SGA4-3}), we see that for large $d$ there is a degree $d$ hypersurface $H_{1,w}$ in $\P_w^N$
intersecting $\cT_w$ transversally and such that $\dim(H_{1,w}\cap(\overline\cT_w-\cT_w))<\dim(H_{1,w}\cap\cT_w)$. If $w\in U$, we may arrange it so that $\delta(w)\in H_{1,w}$. In more detail, if the residue field of $w$ is finite, we apply~\cite[Prop.~3.12]{FedorovMixedChar} with $T_1:=\overline\cT_w-\cT_w$, $T':=\emptyset$, $T:=\cT_w$, $F:=\{\delta(w)\}$ if $w\in U$ and $F:=\emptyset$ otherwise. (We remark that the required hypersurface in~\cite[Prop.~3.12]{FedorovMixedChar} exists for all large enough $d$ as is clear from its proof.) If the residue field of $w$ is infinite, we write $\cT_0:=\cT_w$, stratify $\overline\cT_w-\cT_w$ as $\bigsqcup_{i=1}^n\cT_i$, where $\cT_i$ are smooth, and for $i=0,\ldots,n$ apply~\cite[exp.~XI, thm.~2.1(ii)]{SGA4-3}) with $V':=\cT_i$, $V$ being the closure of $V'$, $P:=\delta(w)$ if $w\in U$ and $P$ being any rational point of $\P_W^N$ otherwise.

Since $W$ is semilocal and affine, we can lift the hypersurfaces $H_{1,w}$, where $w$ ranges over the (finitely many) closed points of~$W$, to a hypersurface $H_1\subset\P_W^N$ containing $\delta(U)$.

Next, we find a hypersurface $H_2\subset\P_W^N$ intersecting $H_1\cap\cT$ transversally, containing $\delta(U)$, and such that for all closed points $w$ in $W$ the dimension of $H_{2,w}\cap H_{1,w}\cap(\overline\cT_w-\cT_w)$ is smaller than dimension of $H_{2,w}\cap H_{1,w}\cap\cT_w$. Repeating this procedure, we find a closed subscheme $\widetilde W\subset\overline\cT$ finite over $W$ (being projective and quasi-finite) and such that $\widetilde W\supset\delta(U)$. Moreover, the dimensional inequalities show that $\widetilde W$ does not intersect the infinity divisor $\overline\cT-\cT$, so $\widetilde W\subset\cT$. By construction, the projection $\widetilde W\to W$ is \'etale at the closed points. Since $W$ is semilocal, $\widetilde W\cap\cT$ is \'etale over $W$. (Cf.~\cite[Prop.~4.1]{FedorovPanin} and~\cite[Lm.~4.3]{PaninNiceTriples}.)
\end{proof}

The following proposition is in the vein of~\cite[Prop.~5.1]{PaninStavrovaVavilov}, \cite[Thm.~4.1]{PaninNiceTriples}, and~\cite[Lm.~5.1]{CesnaviciusGrSerre}.

\begin{proposition}\label{pr:EquatingGroups}
Let $W$ be a semilocal normal Noetherian affine scheme. Assume that $\bG_1$ and $\bG_2$ are reductive $W$-group schemes of the same type and $\bP_i^\pm\subset\bG_i$ are pairs of opposite parabolic $W$-subgroup schemes of the same type. Let $\bT_i\subset\bP^-_i\cap\bP^+_i$ be maximal $W$-tori, $i=1,2$. Assume that $U\subset W$ is a~closed subscheme and $\iota\colon(\bG_1)_U\to(\bG_2)_U$ is an isomorphism such that $\iota$ sends $(\bP_1^\pm)_U$ isomorphically onto $(\bP_2^\pm)_U$ and $(\bT_1)_U$ isomorphically onto $(\bT_2)_U$.

Then there is a finite \'etale cover $\pi\colon\widetilde W\to W$ with a section $\delta\colon U\to\widetilde W$ and an isomorphism $\tilde\iota\colon(\bG_1)_{\widetilde W}\to(\bG_2)_{\widetilde W}$ such that $\delta^*\tilde\iota=\iota$ and $\tilde\iota$ sends $(\bP_1^\pm)_{\widetilde W}$ isomorphically onto $(\bP_2^\pm)_{\widetilde W}$ and $(\bT_1)_{\widetilde W}$ isomorphically onto $(\bT_2)_{\widetilde W}$.
\end{proposition}
\begin{proof}
Let $\bG_1^{\ad}$ be the adjoint group scheme of $\bG_1$ and let $\bT_1^{\ad}$ be the image of $\bT_1$ in $\bG_1^{\ad}$.
Let $\bA$ be the group scheme of automorphisms of $\bG_1$ preserving $\bP^\pm_1$ and $\bT_1$. Then it is easy to derive from~\cite[exp.~XXIV, prop.~2.1]{SGA3-3} and~\cite[exp.~XXIV, thm.~1.3(ii)]{SGA3-3} that $\bA$ is an extension of an \'etale locally constant group scheme by $\bT^{\ad}_1$.

Let $\widetilde I$ be the scheme of isomorphisms $\bG_1\to\bG_2$ taking $\bP^\pm_1$ isomorphically onto $\bP^\pm_2$ and $\bT_1$ isomorphically onto $\bT_2$. Then $\widetilde I$ is an $\bA$-torsor over $W$ (use~\cite[exp.~XXVI, prop.~1.3]{SGA3-3}). Thus $I:=\widetilde I/\bT_1^{\ad}$ is \'etale locally constant over $W$. Note that $\iota$ gives a section $\delta\colon U\to\widetilde I$. Let $\delta'\colon U\to I$ be the composition of $\delta$ with the projection to $I$.

By~\cite[exp.~X, cor.~5.14]{SGA3-2} the connected components of $I$ are finite over $W$. They are also \'etale over $W$. Note that $U':=\delta'(U)$ intersects only finitely many components, denote their union by $I'$, so that $I'$ is finite and \'etale over $U$ and $U'\subset I'$. Then $\delta$ decomposes as $\delta''\circ\delta'$, where $\delta''\colon U'\to\widetilde I$. Applying Proposition~\ref{pr:TorTorsors}\eqref{item:torus3} to the $\bT_1^{\ad}$-torsor $I'\times_I\widetilde I\to I'$ and its section $\delta''$, we get a subscheme $\widetilde W\subset I'\times_I\widetilde I$ finite and \'etale over $W$ and containing $\delta''(U')=\delta(U)$. Since $\widetilde W$ is a subscheme of $\widetilde I$, we get an isomorphism $\tilde\iota\colon(\bG_1)_{\widetilde W}\to(\bG_2)_{\widetilde W}$ sending $(\bP_1^\pm)_{\widetilde W}$ isomorphically onto $(\bP_2^\pm)_{\widetilde W}$ and $(\bT_1)_{\widetilde W}$ isomorphically onto $(\bT_2)_{\widetilde W}$. By construction, $\delta^*\tilde\iota=\iota$.
\end{proof}

\subsection{Preparation to the descent} We now show that in Proposition~\ref{pr:PrepDescend} we may assume that the group scheme $\bG'$ is the pullback of $\bG$ under the projection $C\to\Spec\cO_{X,\bx}$. We also change our relative curve $C$ in such a way that the torsor can be descended to $\A^1_{\cO_{X,\bx}}$. We need one more preliminary proposition.

\begin{proposition}\label{pr:prepare0}
For
\begin{itemize}
  \item a semilocal Noetherian ring $R$;
  \item an affine smooth $R$-scheme $C$ of pure relative dimension one;
  \item closed subschemes $T\subset Y\subset C$ such that $T$ is $R$-finite and $Y$ is $R$-quasi-finite; and
  \item a section $\Delta\in T(R)$;
\end{itemize}
there are
\begin{itemize}
     \item an \'etale $R$-morphism $\widetilde C\to C$ such that $T\times_C\widetilde C$ is $R$-finite;
     \item a lift $\widetilde\Delta\in\widetilde C(R)$ of $\Delta$;
     \item an \'etale $R$-morphism $\widetilde C\to W$, where $W$ is an open subset of $\A^1_R$, that maps $Y\times_C\widetilde C$ isomorphically onto a closed subscheme $Y'\subset W$ such that $Y\times_C\widetilde C\simeq Y'\times_W\widetilde C$.
\end{itemize}
\end{proposition}
\begin{proof}
Applying~\cite[Lm.~6.1]{CesnaviciusGrSerre} we get a finite $R$-morphism $\phi\colon C'\to C$ such that $\phi$ is \'etale in a~neighborhood of $T':=T\times_CC'$, $\Delta$ lifts to $\Delta'\in T'(R)$, and for all closed points $x$ of $\Spec R$ we have
\begin{equation}\label{eq:cond}
    \#\{z\in T'_x\colon [k(z):k(x)]=d\}<\#\{z\in \A^1_x\colon [k(z):k(x)]=d\}\text{ for every }d\ge1.
\end{equation}
Note that $T'$ is also $R$-finite. Let $C''$ be a Zariski neighborhood of $T'$ in $C'$ such that $\phi|_{C''}$ is \'etale. It is enough to prove the proposition with $C$, $Y$, and $T$ replaced with $C''$, $Y\times_CC''$, and $T'$. Thus, we may assume from the beginning that condition~\eqref{eq:cond} is satisfied with $T$ instead of $T'$. It remains to apply~\cite[Variant~4.8]{GuoLiuGrSerre}.
\end{proof}

The following is an analogue of~\cite[Prop.~6.5]{CesnaviciusGrSerre}.
\begin{proposition}\label{pr:prepare}
Let $X$ be an integral affine scheme smooth and of relative positive dimension over a semilocal Dedekind domain. Let $\bG$ be a reductive $X$-group scheme with opposite parabolic $X$-subgroup schemes $\bP_\pm$. Assume that the $X$-group scheme $\bP_-\cap\bP_+$ has a maximal torus. Let $T\subset Y\subset X$ be closed subsets such that $Y$ is fiberwise of positive codimension over the Dedekind domain and $T$ is of codimension at least two in $X$. Let $\cE$ be a $\bG$-torsor over $X$ such that there is a unipotent chain of $\bG$-torsors over $X-T$ connecting $\bG_{X-T}$ to $\cE_{X-T}$ and assume that the~chain is trivial over $X-Y$. Fix a finite set of points $\bx\subset X$ and set $R:=\cO_{X,\bx}$. Then there are

\stepzero\noindstep\label{item:a} a smooth integral affine $R$-scheme $C$ of pure relative dimension one;

\noindstep a section $\Delta\in C(R)$;

\noindstep a $\bG_R$-torsor $\cE'$ over $C$ such that $\Delta^*\cE'\simeq\cE_R$;

\noindstep\label{item:d} closed subschemes $\widetilde Z\subset\widetilde Y\subset C$ such that $\widetilde Z$ is $R$-finite and $\widetilde Y$ is $R$-quasi-finite;

\noindstep\label{item:e} a unipotent chain of $\bG$-torsors over $C-\widetilde Z$ that connects $\bG_{C-\widetilde Z}$ to $\cE'_{C-\widetilde Z}$ and such that the chain is trivial over $C-\widetilde Y$ with respect to $\bP_\pm$;

\noindstep an \'etale $R$-morphism $C\to W$, where $W$ is an open subset of $\A^1_R$, that maps $\widetilde Y$ isomorphically onto a closed subscheme $Y'\subset W$ with $\widetilde Y\simeq Y'\times_WC$.
\end{proposition}
\begin{proof}
Let $\bT$ be a maximal torus of $\bP_-\cap\bP_+$. Consider the data $Z\subset C\to\Spec R$, $\Delta$, $\bT'\subset\bP'_-\cap\bP'_+\subset\bG'$, and~$\cE'$ provided by Proposition~\ref{pr:PrepDescend}. We also get a unipotent chain over $C-Z$ trivial away from an $R$-quasi-finite closed subscheme $Y_0\subset C$; replacing $Z$ with $Z\cup\Delta(\Spec R)$, and $Y_0$ with $Y_0\cup\Delta(\Spec R)$, we may assume that $\Delta(\Spec R)\subset Z$.

Applying Proposition~\ref{pr:EquatingGroups} to the semilocalization of $Z$ in~$C$ and spreading out, we find an open subscheme $C'\subset C$ containing $Z$, a finite \'etale morphism $C''\to C'$, a section $\Delta'\colon\Spec R\to C''$ lifting $\Delta$, and an isomorphism $\tilde\iota\colon\bG_{C''}\to\bG'_{C''}$ such that $(\Delta')^*\tilde\iota$ is the isomorphism of item~\eqref{item:d} of Proposition~\ref{pr:PrepDescend}. (Here $\bG_{C''}$ is the pullback of $\bG$ via the composition $C''\to C\to\Spec R\to X$.) Replacing $C''$ by a connected component, we may assume that $C''$ is integral.

Put $\cE'':=\cE'_{C''}$. We view $\cE''$ as a $\bG$-torsor using the isomorphism $\tilde\iota$. Now setting $\widetilde Z:=Z\times_CC''$, $\widetilde Y:=Y_0\times_CC''$, pulling back the unipotent chain to $C''-\widetilde Z$, and renaming $\cE''\rightsquigarrow\cE'$, $C''\rightsquigarrow C$, and $\Delta'\rightsquigarrow\Delta$ , we obtain items~\eqref{item:a}--\eqref{item:e} of the proposition. Moreover, $\Delta(R)\subset\widetilde Z$.

Further, applying Proposition~\ref{pr:prepare0} with $T:=\widetilde Z$ and $Y:=\widetilde Y$, we see that there are an \'etale $R$-morphism $\widetilde C\to C$ such that $\widetilde Z\times_C\widetilde C$ if $R$-finite; a lift $\widetilde\Delta\in\widetilde C(R)$ of $\Delta$; and an \'etale $R$-morphism $\widetilde C\to W$, where $W$ is an open subset of $\A^1_R$, that maps $\widetilde Y\times_C\widetilde C$ isomorphically onto a closed subscheme $Y'\subset W$ such that $\widetilde Y\times_C\widetilde C\simeq Y'\times_W\widetilde C$.

It remains to pullback the unipotent chain to $\widetilde C-\widetilde C\times_C\widetilde Z$ and to rename $\widetilde C\times_C\widetilde Z\rightsquigarrow\widetilde Z$, $\widetilde C\times_C\widetilde Y\rightsquigarrow\widetilde Y$,
$\cE'_{\widetilde C}\rightsquigarrow\cE'$, $\widetilde\Delta\rightsquigarrow\Delta$, and
$\widetilde C\rightsquigarrow C$.
\end{proof}

\subsection{Descending to \texorpdfstring{$\A^1$}{A1}}\label{sect:DescendA1}
The goal of this section is to descend the data of Proposition~\ref{pr:prepare} to $\A^1_R$. We start with the descend statement for unipotent chains. Recall from~\cite[Ch.~3, Def.~1.3]{MorelVoevodsky} and~\cite[Def.~3]{FedorovNisnevich}) that an \emph{elementary distinguished square\/} is a Cartesian diagram of schemes:
\begin{equation}\label{eq:ElemDistSq}
\begin{CD}
W @>>> V\\
@VVV @VV p V\\
U @>j>> Y,
\end{CD}
\end{equation}
where $p$ is \'etale, $j$ is an open embedding, and $p^{-1}(Y-U)\to Y-U$ is an isomorphism, where we equip the corresponding closed subsets with the reduced scheme structures.

\begin{lemma}\label{lm:Descent}
Consider an elementary distinguished square~\eqref{eq:ElemDistSq}. Assume that $\bG$ is an affine $Y$-group scheme.

(i) Let $\cE'$ and $\cE''$ be $\bG$-torsors over $U$ and $V$ respectively. Given an isomorphism between $\bG$-torsors $\cE'|_U\xrightarrow{\sigma}\cE''|_V$, there is a $\bG$-torsor $\cE$ over $Y$ with isomorphisms $\cE|_U\xrightarrow{\sigma'}\cE'$ and $\cE|_V\xrightarrow{\sigma''}\cE''$ such that $\sigma\circ\sigma'|_W=\sigma''|_W$. The triple $(\cE,\sigma',\sigma'')$ is unique up to a unique isomorphism.

(ii) Let $\bP'\subset\bG|_U$ and $\bP''\subset\bG|_V$ be closed subgroup schemes such that $\bP'|_W=\bP''|_W$. Then there is a unique $Y$-subgroup scheme $\bP\subset\bG$ such that $\bP|_U=\bP'$ and $\bP|_V=\bP''$.

(iii) Let $\bP\subset\bG$ be a closed $Y$-subgroup scheme and $\cE$ be a $\bG$-torsor over $Y$. Let $(\cF',\phi')$ be a $\bP|_U$-reduction of $\cF|_U$ and $(\cF'',\phi'')$ be a $\bP|_V$-reduction of $\cE|_V$. Assume further that $(\cE',\phi')|_W\simeq(\cE'',\phi'')|_W$. Then there is a unique $\bP$-reduction $(\cF,\phi)$ of $\cE$ such that $(\cF,\phi)|_U\simeq(\cF',\phi')$ and $(\cF,\phi)|_V\simeq(\cF'',\phi'')$.
\end{lemma}

\begin{proof}
(i) This is standard; see for example the first part of the proof of~\cite[Ch.~3, Prop.~1.4]{MorelVoevodsky} or~\cite[Sect.~5.4]{FedorovPanin}.

(ii) Similarly to (i), we can glue the affine group schemes $\bP'$ and $\bP''$ to obtain a $Y$-group scheme $\bP$ with an isomorphism $\bP|_U\simeq\bP'$ and $\bP|_V\simeq\bP''$. Next we glue the embeddings $\bP'\to\bG|_U$ and $\bP''\to\bG|_V$ to a homomorphism $\bP\to\bG$. This homomorphism is a closed embedding as this is an \'etale local property. Thus, we can view $\bP$ as a subgroup scheme of $\bG$. By construction $\bP$ satisfies the requirements. The uniqueness is also clear.

(iii) Using part (i) we glue $\cF'$ and $\cF''$ to obtain a $\bP$-torsor $\cF$ over $Y$. We glue $\phi'$ and $\phi''$ to obtain an isomorphism $\phi\colon\cF\times^\bP\bG\to\cE$.
\end{proof}

\begin{remark}
Under slightly different assumptions that are still satisfied in our situation, the required gluing statements should follow from~\cite[Prop.~4.2.1]{CesnaviciusProblems}. We prefer to use techniques of~\cite{MorelVoevodsky} to emphasize the Nisnevich local nature of our constructions.
\end{remark}

\begin{proposition}\label{pr:DescendChain}
Consider an elementary distinguished square~\eqref{eq:ElemDistSq}. Assume that $\bG$ is a reductive $Y$-group scheme with opposite parabolic subgroup schemes $\bP_\pm\subset\bG$. Let
\[
\bG_V,\cE_2,\ldots,\cE_n, \bP_1,\ldots,\bP_{n-1},\tau_1,\ldots,\tau_{n-1}
\]
be a unipotent chain of $\bG_V$-torsors trivial over $W$ (with respect to $(\bP_\pm)_W$). Then the chain descends to a unipotent chain over~$Y$ that is trivial over $U$ (with respect to $(\bP_\pm)_U$).
\end{proposition}
\begin{proof}
    Since the chain is trivial over $W$, we have trivializations $s_i\colon\bG_W\to(\cE_i)_W$ identifying $(\bP_i)_W$ with $(\bP_\pm)_W$ ($i=2,\ldots,n$). Using trivializations $s_i$ and Lemma~\ref{lm:Descent}(i), we descend the torsors $\cE_i$ to $\bG$-torsors $\cE'_i$ over $Y$ with distinguished trivializations over $U$ and isomorphisms $(\cE'_i)_V\xrightarrow{\sigma_i}\cE_i$ such that $(\sigma_i)_W$ identifies the induced trivialization of $(\cE'_i)_W$ with $s_i$. Assume that $(\bP_i)_W$ corresponds to $(\bP_-)_W$ under the trivialization (the case of $(\bP_+)_W$ is completely similar). Let $\bP''_i\subset\Aut(\cE'_i)_U$ correspond to $\bP_-$ under the distinguished trivialization of $\cE'_i$. Then under the isomorphism $\Aut(\cE'_i)_W\xrightarrow{\sim}\Aut(\cE_i)_W$ induced by $(\sigma_i)_W$,  $(\bP''_i)_W$ corresponds to $(\bP_i)_W$ as they both correspond to $(\bP_-)_W$ under the trivializations. By Lemma~\ref{lm:Descent}(ii) there is a unique $\bP'_i\subset\Aut(\cE'_i)$ such that $(\bP'_i)_U=\bP''_i$ and $(\bP'_i)_V=\bP_i$.

    Let $\tau''_i$ be the reduction of $\Iso(\cE'_i,\cE'_{i+1})_U$ to $(\Ru\bP'_i)_U$ corresponding to the standard reduction via the distinguished trivializations of $\cE'_i$ and $\cE'_{i+1}$. By construction, $(\tau''_i)_W=(\tau_i)_W$, where we identify $\Iso(\cE'_i,\cE'_{i+1})_W$ with $\Iso(\cE_i,\cE_{i+1})_W$ using $\sigma_i$ and $\sigma_{i+1}$, so by Lemma~\ref{lm:Descent}(iii) we obtain a reduction $\tau'_i$ of $\Iso(\cE'_i,\cE'_{i+1})$ to $\Ru\bP'_i$ extending the standard reduction on $U$. Now
    \[
        \bG,\cE'_2,\ldots,\cE'_n, \bP'_1,\ldots,\bP'_{n-1},\tau'_1,\ldots,\tau'_{n-1}
    \]
    is the required chain. The chain is trivial over $U$ by construction.
\end{proof}

\begin{proposition}\label{pr:A1}
Let $X$ be an integral affine scheme smooth of positive relative dimension over a semilocal Dedekind domain. Let $\bG$ be a reductive $X$-group scheme with opposite proper parabolic $X$-subgroup schemes $\bP_\pm$. Assume that the $X$-group scheme $\bP_-\cap\bP_+$ has a maximal torus. Let $T\subset Y\subset X$ be closed subsets such that $Y$ is fiberwise of positive codimension over the Dedekind domain and $T$ is of codimension at least two in $X$. For a finite set of points $\bx\subset X$, set $R:=\cO_{X,\bx}$. Let $\cE$ be a $\bG$-torsor over $X$ such that there is a unipotent chain of $\bG$-torsors over $X-T$ connecting $\bG_{X-T}$ to $\cE_{X-T}$ and assume that the chain is trivial over $X-Y$. Then there are

\stepzero\noindstep
a closed $R$-finite subscheme $Y\subset\A^1_R$;

\noindstep a $\bG$-torsor $\cE'$ over $\A^1_R$ trivial away from $Y$; and

\noindstep a section $\Delta\in\A^1_R(R)$ such that $\Delta^*\cE'\simeq\cE_R$.
\end{proposition}
\begin{proof}
Set $U:=\Spec R$. Consider the data $\widetilde Z\subset\widetilde Y\subset C\to W\supset Y, \Delta, \cE'$ provided by Proposition~\ref{pr:prepare}. Let $Z\subset Y$ be the image of $\widetilde Z$ in $W$ so that $Z$ is $R$-finite. Let $\bG_{C-\widetilde Z},\cE_1,\ldots,\cE_{n-1},\cE'_{C-\widetilde Z},\bP_1,\ldots$ be the unipotent chain trivial over $C-\widetilde Y$. We fix the corresponding trivializations of $(\cE_i)_{C-\widetilde Y}$ and $\cE'_{C-\widetilde Y}$.

We have the following elementary distinguished square
\begin{equation*}
\begin{CD}
C-\widetilde Y @>>> C\\
@VVV @VVV \\
W-Y @>>>W.
\end{CD}
\end{equation*}
We use this square and the trivialization of $\cE'$ over $C-\widetilde Y$ to glue $\cE'$ with the trivial $\bG$-torsor over $W-Y$ (see~\cite[prop.~2.6]{ColliotTeleneOjanguren}). We obtain a $\bG$-torsor $\cE''$ over $W$. By Proposition~\ref{pr:DescendChain} applied to the elementary distinguished square
\begin{equation*}
\begin{CD}
C-\widetilde Y @>>> C-\widetilde Z\\
@VVV @VVV \\
W-Y @>>>W-Z
\end{CD}
\end{equation*}
the chain $\bG_{C-\widetilde Z},\cE_1,\ldots,\cE_{n-1},\cE'_{C-\widetilde Z},\bP_1,\ldots$ descends to $W-Z$. It follows from the construction in Proposition~\ref{pr:DescendChain}, that the obtained unipotent chain connects $\bG_{W-Z}$ with $\cE''_{W-Z}$.

Next, by~\cite[Lemma~3.5(ii)]{FedorovMixedChar} (which generalizes immediately from the local to the semilocal case) there is an $R$-finite closed subscheme $Y'\subset W$ such that $Z\subset Y'$ and $W-Y'$ is affine. Thus, by Lemma~\ref{lm:ChainTrivAff} $\cE''$ is trivial over $W-Y'$. Gluing $\cE''$ with the trivial $\bG$-torsor over $\A^1_R-Y'$, we obtain a $\bG$-torsor $\cE'''$ over $\A^1_R$.

It remains to replace $\Delta$ with the composition of $\Delta$ and $C\to \A^1_R$ and rename $\cE'''\rightsquigarrow\cE'$, $Y'\rightsquigarrow Y$.
\end{proof}

\section{Torsors over \texorpdfstring{$\A^1$}{A1} and completion of the proof of Theorem~\ref{th:Main2}}\label{sect:TorsorsA1}
The goal of this section is to prove Theorem~\ref{th:A1} and to complete the proof of Theorem~\ref{th:Main2}.

\subsection{} Let $G$ be a semisimple group scheme over a field $k$. Let $\phi\colon G^{\,\sc}\to G$ be the simply-connected central cover. Recall the notion of a topologically trivial torsor~\cite[Def.~2.1]{FedorovGrSerreNonSC}.

\begin{definition}
A Zariski locally trivial $G$-torsor $E$ over the projective line $\P^1_k$, where $k$ is a field, is called \emph{topologically trivial\/} if there is a Zariski locally trivial $G^{\,\sc}$-torsor $E^{\sc}$ over $\P^1_k$ such that $\phi_*E^{\sc}\simeq E$.
\end{definition}

Recall that a semisimple group scheme of adjoint type can be written as the product of Weil restrictions of simple group schemes along finite connected \'etale covers (see~\cite[exp.~XXIV, prop.~5.10]{SGA3-3}). We start with the following theorem.

\begin{theorem}\label{th:MainThm2}
Let $U$ be a connected affine semilocal scheme. Let $\bG$ be a~reductive group scheme over~$U$ with center $\bZ$; write
\[
    \bG^{\ad}:=\bG/\bZ\simeq\prod_{i=1}^r\bG^i,
\]
where $\bG^i$ is the Weil restriction of a simple $U_i$-group scheme $\overline\bG^i$ along a finite \'etale morphism $U_i\to U$. Let $Z\subset\A^1_U$ be a closed subscheme finite over $U$. Let $\cG$ be a $\bG$-torsor over $\P^1_U$ such that its restriction to $\P^1_U-Z$ is trivial and such that for all closed points $u\in U$ the $\bG_u^{\ad}$-torsor $(\cG_{\P^1_u})/\bZ_u$ is topologically trivial.

Let $Y\subset\A^1_U$ be a closed subscheme finite and \'etale over $U$. Assume that $Y\cap Z=\emptyset$. Assume further that for each $i=1,\ldots,r$ there is an open and closed subscheme $Y^i\subset Y\times_UU_i$ satisfying two properties: {\rm(}i{\rm)} the pullback of $\overline\bG^i$ to $Y^i$ is isotropic, and {\rm(}ii{\rm)} for every closed point $v\in U_i$ such that $\overline\bG^i_v$ is isotropic we have $\Pic(\P^1_v-Y^i_v)=0$. Finally, assume that the relative line bundle $\cO_{\P^1_U}(1)$ trivializes on $\P^1_U-Y$.

Then the restriction of $\cG$ to $\P^1_U-Y$ is also trivial.
\end{theorem}
\begin{proof}
  If $U$ is a scheme over a field, then this is~\cite[Thm.~6]{FedorovGrSerreNonSC}. By inspection, the proof goes through in the general case except that the reference to~\cite[Prop.~2.12]{FedorovGrSerreNonSC} should be replaced with the reference to~\cite[Lm.~8.3]{CesnaviciusGrSerre}. Alternatively, one can use~\cite[Prop.~5.3.6]{CesnaviciusProblems}.
\end{proof}

Next, we need a lemma.

\begin{lemma}\label{lm:isotropic}
Let $U_i$, $U$, and $\overline{\bG}^i$ be as in the formulation of the theorem, where $1\le i\le r$. Let $T\subset\A^1_U$ be a closed subscheme finite over $U$. Then for every $i$ there is a finite \'etale $U_i$-scheme $Y$ such that $\overline{\bG}^i_Y$ is isotropic and the following condition is satisfied:

(*) if $v$ is a closed point of $U_i$ such that $\overline{\bG}^i_v$ is isotropic, then the $k(v)$-degrees of the points of $Y_v$ are coprime.

Moreover, this $Y$ can be chosen so that there is a closed $U$-embedding $Y\to\A^1_U-T$, where $Y$ is viewed as a $U$-scheme via the composition $Y\to U_i\to U$.
\end{lemma}
\begin{proof}
Let $\cP$ be the $U_i$-scheme classifying the proper parabolic subgroup schemes of $\overline{\bG}^i$. By~\cite[exp.~XXVI, cor.~3.6]{SGA3-3} $\cP$ is smooth and projective over $U_i$. Thus, for some $N>0$ we have a $U_i$-embedding $\cP\to\P_{U_i}^N$. Let $v$ be a closed point of $U_i$. If $\overline\bG^i_v$ is isotropic, then we have a $k(v)$-rational point $\widetilde v$ on the fiber $\cP_v$.

Next, we argue as in the proof of Proposition~\ref{pr:TorTorsors}\eqref{item:torus3}. Let $\cP'\subset\cP$ be a connected component. Using Bertini's Theorem (see~\cite{PoonenOnBertini} and~\cite[exp.~XI, thm.~2.1(ii)]{SGA4-3}), we see that for large $d$ there is a degree~$d$ hypersurface $H_{1,v}$ in $\P_v^N$ intersecting~$\cP'_v$ transversally and containing all the point $\widetilde v$ that belong to~$\cP'$. We lift the hypersurfaces $H_{1,v}$, where $v$ ranges over the closed points of $U_i$, to a hypersurface $H_1\subset\P_{U_i}^N$. Thus, $H_1$ intersects $\cP'$ transversally and contains all the points $\tilde v$ that belong to $\cP'$. Next, we find a hypersurface $H_2\subset\P_{U_i}^N$ intersecting $H_1\cap\cP$ transversally and containing all the points $\tilde v$ belonging to $\cP'$. Repeating this procedure, we find a closed subscheme $Y'\subset\cP'$ finite and \'etale over $U_i$. Performing this procedure for all connected components of $\cP$ and letting $Y''$ be the union of the subschemes $Y'$, we obtain a closed subscheme $Y''\subset\cP$ that is finite and \'etale over $U_i$ and such that for all closed points~$v$ such that $\overline{\bG}^i_v$ is isotropic, $Y''_v$ has a $k(v)$-rational point. Since $Y''\subset\cP$, $\overline{\bG}^i_{Y''}$ is isotropic. (Cf.~\cite[Prop.~4.1]{FedorovPanin} and~\cite[Lm.~4.3]{PaninNiceTriples}.)

Now we choose a prime number $p_1$ large enough that

    (i) if $u'\in T$ is a point lying over a closed point $u\in U$, then $p_1>[k(u'):k(u)]$;

    (ii) for $n\ge p_1$ and a closed point $u\in U$ such that $k(u)$ is finite, the number of degree $n$ points in $\A^1_{k(u)}$ is larger than the number of points in the $u$-fiber of $Y''$.

Choose a prime number $p_2>p_1$. For every closed point $w\in Y''$ choose a monic polynomial $h_w\in k(w)[t]$ of degree $p_1+p_2$ such that: (i) if $k(w)$ is finite, then $h_w$ is the product of two irreducible polynomials of degrees $p_i$; if $k(w)$ is infinite, then $h_w$ is a separable polynomial having a root in $k(w)$.

Note that $Y''$ is affine and semilocal, since it is finite over $U$. Write $Y''=\Spec A$. Then we can find a~monic polynomial $h\in A[t]$ that reduces to $h_w$ at each closed point $w$. Since $h$ is monic, the scheme $Y:=\Spec A[t]/(h)$ is finite over $Y''$. Since $Y$ has a morphism to $Y''$, we see that $\overline{\bG}^i_Y$ is isotropic.

Condition~(*) is satisfied by the choice of $h_w$.

It remains to check that there is a closed $U$-embedding $Y\to\A^1_U-T$. Let $u\in U$. If $k(u)$ is finite, there is a closed $k(u)$-embedding $Y_u\hookrightarrow\A^1_u-T_u$ because of our conditions (i) and (ii). If $k(u)$ is infinite, there is a closed $k(u)$-embedding $Y_u\hookrightarrow\A^1_u-T_u$ because $Y_u$ is separable over $k(u)$. By the Chinese Remainder Theorem these embeddings can be extended to an embedding $Y\hookrightarrow\A^1_U-T$, since is affine and semilocal.
\end{proof}

\begin{proof}[Proof of Theorem~\ref{th:A1}]
We may assume that~$U$ is connected. Applying an affine transformation to $\A^1_U$, we may assume that $\Delta$ is the horizontal section $\Delta(U)=U\times1$. We can extend the $\bG$-torsor $\cE$ to a $\bG$-torsor $\widetilde\cE$ over $\P^1_U$ by gluing it with the trivial $\bG$-torsor over $\P^1_U-Z$. Let $d$ be the degree of the simply-connected central cover of $\bG^{\ad}$ (see~\cite[Exercise~6.5.2]{ConradReductive}). Consider the morphism $\P^1_\Z\to\P^1_\Z\colon z\mapsto z^d$; let $\psi\colon\P^1_U\to\P^1_U$ be the base change of this morphism. Consider the $\bG$-torsor $\psi^*\widetilde\cE$ over $\P^1_U$. For a closed point $u\in U$ write $\widetilde\cE_u:=\widetilde\cE_{\P_u^1}$. Then by~\cite[thm.~3.8(a)]{GilleTorseurs} the $\bG^{\ad}_u$-torsor $\widetilde\cE_u/\bZ_u$ is Zariski locally trivial. By~\cite[Prop.~2.3]{FedorovGrSerreNonSC} the $\bG^{\ad}_u$-torsor $\psi^*(\widetilde\cE_u/\bZ_u)$ is topologically trivial. Since the morphism $\psi$ has a section over $U\times1$, it is enough to show that $\psi^*\widetilde\cE_{U\times1}$ is trivial. Note that $\psi^*\widetilde\cE$ is trivial over $\P^1_U-\psi^{-1}(Z)$.

Now using Lemma~\ref{lm:isotropic} we construct inductively $U_i$-schemes $Y^i$ satisfying the conditions of the lemma with closed $U$-embeddings $\iota_i\colon Y^i\to\A^1_U$ such that the subschemes $\iota_i(Y^i)$ are disjoint from each other and from $\psi^{-1}(Z)\cup(U\times0)\cup(U\times 1)$.

Take $Y=(U\times 0)\sqcup\bigsqcup_{i=1}^rY^i$. Note that $Y^i$ is an open and closed subscheme of $Y\times_UU_i$ and by construction $\overline\bG^i$ is isotropic over $Y^i$. Let $v$ be a closed point of $U_i$ such that $\overline\bG^i_v$ is isotropic. Then condition~(*) of Lemma~\ref{lm:isotropic} shows that $\Pic(\A^1_v-Y^i_v)=0$. It remains to apply Theorem~\ref{th:MainThm2} to $\psi^{-1}(Z)\subset\A^1_U$, $\psi^*\widetilde\cE$, and $Y$.
\end{proof}

\subsection{Proof of Theorem~\ref{th:Main2}}\label{Sect:Proof}
Let $X$, $\bG$, $\cE$, $T$ and $Y$ be as in the formulation of Theorem~\ref{th:Main2}. Let $\bx\subset X$ be a finite set of points. We need to show that $\cE$ is trivial in a Zariski neighborhood of $\bx$. By~\cite[Exp.~XXVI, thm.~4.3.2(b)]{SGA3-3} $\bP_-\cap\bP_+$ is a Levi $X$-subgroup scheme of $\bP_\pm$, in particular, it is reductive. Thus, by~\cite[Exp.~XIV, cor.~3.20]{SGA3-3} (see also the footnote) we may replace $X$ with an affine Zariski neighborhood of $\bx$ so that $\bP_-\cap\bP_+$ contains a maximal torus.

Set $R:=\cO_{X,\bx}$. We can apply Proposition~\ref{pr:A1} to obtain a $\bG_R$-torsor $\cE'$ over $\A^1_R$ trivial away from an $R$-finite subscheme and a section $\Delta$ of $\A_R^1\to\Spec R$ such that $\Delta^*\cE'\simeq\cE_R$. By Theorem~\ref{th:A1}, $\cE_R\simeq\Delta^*\cE'$ is trivial. Spreading out, we see that $\cE$ is trivial in a Zariski neighborhood of $\bx$. \qed

\bibliographystyle{../../alphanum}
\bibliography{../../RF}
\end{document}